\documentclass[reqno, 11pt]{amsart}
\usepackage[top=4cm, bottom=3cm, left=3.2cm, right=3.2cm]{geometry}
\usepackage{enumerate}
\usepackage{amssymb,amsmath}
\usepackage{xcolor}

\newtheorem{thm}{Theorem}

\newtheorem{lem}{Lemma}
\newtheorem{prop}{Proposition}
\theoremstyle{definition}
\newtheorem{defn}{Definition}

\theoremstyle{remark}

\allowdisplaybreaks

\numberwithin{equation}{section} \numberwithin{lem}{section}
\numberwithin{thm}{section} \numberwithin{prop}{section}
\numberwithin{cor}{section} \numberwithin{rem}{section}
\numberwithin{defn}{section}
\numberwithin{assume}{section}

	\catcode`@=11 \@addtoreset{equation}{section} \catcode`@=12
\title[Global existence of solutions]{Global existence of martingale solutions to stochastic keller-segel system with degenerate diffusion}
\author{Jinhuan Wang$^{\,1}$, Qian Li$^{\,1}$, and Hui Huang$^{\,2}$}
\thanks{The work of J. Wang is patially supported by National Natural Science Foundation of China Grants No. 12171218, LiaoNing Revitalization Talents Program Grant No. XLYC2007022.}
\thanks{Corresponding author: Hui Huang}
\begin{document}

\maketitle
\begin{center}
{\footnotesize
1-School of Mathematics, Liaoning University, Shenyang, 110036, P. R. China \\
email:  wjh800415@163.com; liqian990720@163.com\\
 \smallskip
2-Department of Mathematics and Scientific Computing, Karl-Franzens-Universit\"{a}t Graz, 8010 Graz, Austria
\\
email: hui.huang@uni-graz.at
}
	\end{center}
\maketitle
 \date{}     
\begin{abstract}
In this paper, we study the stochastic degenerate Keller–Segel system perturbed by linear multiplicative noise in a bounded domain $\mathcal{O}$. We establish the global existence of martingale solutions for this model with any nonnegative initial data in $H_{2}^{-1}(\mathcal{O})$. The main challenge in proving the existence of solutions arises from the degeneracy of the porous media diffusion and the lack of coercivity in the nonlinear chemotactic term. To overcome these difficulties, we construct a solution operator and apply the Schauder fixed point theorem within the variational framework.
\end{abstract}

\section{Introduction}
In this paper, we study the following stochastic  Keller-Segel equations with degenerate diffusion
\begin{eqnarray}\label{equ}
\begin{cases}	
	d\rho  = [\Delta {(|\rho| ^{m-1}\rho)} - \chi \nabla  \cdot (\rho \nabla\Phi*\rho) ]dt + \sigma \rho dW(t),&x\in\mathcal{O},t\ge 0,\\	\rho (x,0) = \rho_{0},&x\in\mathcal{O},t= 0,
 \end{cases} 
\end{eqnarray}
where the diffusion exponent $m \geq 3$, the constant $\sigma > 0$, and $\mathcal{O}$ is a bounded domain in $\mathbb{R}^d$, with $d = 2,3$, and a smooth boundary $\partial\mathcal{O}$.
Here, $\rho(x,t)$ represents the bacterial density, while the interaction potential $\Phi: \mathbb{R}^d \to \mathbb{R}$ is either the Newtonian or Bessel potential. The term $-\rho \nabla\Phi * \rho$ represents the chemotactic flux, which describes the transport of bacteria in the direction of the chemical gradient. This essentially produces $- \chi \nabla  \cdotp (\rho \nabla\Phi*\rho )$, illustrating cross-diffusive effects in the model, where the positive constant $\chi$ represents the chemotactic sensitivity.
\par
The linear operator $\Delta$ denotes the Laplacian with Dirichlet (or Neumann) boundary condition.  It is well
known that $-\Delta$ is self-adjoint positive and anti-compact. We assume that $\left \{ e _{k}  \right \}_{k\in \mathbb{N}}$ is the orthonormal basis in $L^2{(\mathcal{O})}$ composed of eigenfunctions of $-\Delta$. We denote by $\left \{ \lambda _{k}  \right \}_{k\in \mathbb{N}}$ the corresponding
sequence of eigenvalues\begin{align*}
-\Delta e _{k}=\lambda _{k} e _{k}\text{,\quad}k\in \mathbb{N}.
\end{align*}  We shall consider a cylindrical Wiener process  $W(t)$ on $L^2{(\mathcal{O})}$  of the form 
\begin{align*}
W(t)=\sum_{k=1}^{\infty }   e_{k} \beta _{k}\left( t\right),~~t \ge0,
\end{align*}
 where $\left \{ \beta _{k}  \right \}$  is a sequence of mutually independent standard Brownian motions on a right-continuous complete  filtration probability space $(\Omega, \mathcal{F}, \left \{\mathcal{F}_{t} \right \} _{t\ge 0}, \mathbb{P}  )$  and $\rho dW(t)$ is a linear multiplicative  noise.
 
 Since the bacteria density $\rho (x,t)$  has to be non-negative, the initial condition  $\rho_{0}$ has to be non-negative. Then the model possesses  the  initial data $\rho_0 \in L^{2}(\Omega ,\mathcal{F}_{0} ,H_{2}^{-1}(\mathcal{O} ))$, which is a random variable over $(\Omega ,\mathcal{F},\left \{\mathcal{F}_{t} \right \} _{t\ge 0}  ,\mathbb{P}  )$ and satisfies
(i) $\rho_{0}\ge0$;
(ii) $\rho_{0}$ is $\mathcal{F}_0$–measurable;
(iii) $\mathbb{E}\left [\|\rho_{0}\|_{L^{m+1}}^{m+1}\right]<\infty$.

 Chemotaxis refers to the directed movement of cells or organisms in response to a chemical gradient. A variety of motion cell types exhibit the phenomenon of chemotaxis. The theoretical and mathematical model of chemotaxis can be traced back to the research of Patlak  \cite{patlak1953random}
   in the 1950s and Keller and Segel \cite{1970Initiation} in the 1970s. Let us consider the following classical degenerate Keller-Segel system with a nonlinear diffusion 
   \begin{align}\label{equ2}
   	d\rho  = [\Delta {\rho ^m} - \chi \nabla  \cdot (\rho \nabla \Phi*\rho)]dt, ~x\in\mathbb{R}^d,t> 0. 
   \end{align}
   The equation (\ref{equ2}) can be obtained as the continuum limit of a many-particle system  \cite{2006An,1990Large} and has various applications in physics and biology, depending on the selected interaction kernel and diffusion parameter $m \ge1$. These applications include self-organization in chemotactic movement \cite{2008Infinite,1970Initiation,2006Transport}, biological swarms \cite{2013Stationary, 2006A} and cancer invasion \cite{2014Mathematical,gerisch2008mathematical}. When the interaction kernel $\Phi$ is a Newtonian or Bessel potential, the equation (\ref{equ2}) is well known as the Keller–Segel chemotaxis model, for which numerous significant findings have been reported. A key characteristic of the Keller–Segel model is the  critical mass  for the existence of solutions. When $d=2$ and $m=1$, the equation (\ref{equ2}) indicates the linear random motion in \cite{ 2008Infinite,blanchet2006two,dolbeault2004optimal}, the solution of this model has a critical mass of $8\pi$. When $d\ge 3$ and $m>1$, the behavior of solutions were studied in \cite{bian2013dynamic,blanchet2009critical,chen2012multidimensional,chen2014exact,luckhaus2006large,sugiyama2006global,sugiyama2007time,sugiyama2006globale2}. In particular, when $d\ge3$ and $m=2-\frac{2}{d}\in(1,2)$, Blanchet et~al. \cite{blanchet2009critical} proved that there is a critical mass such that the solution to (\ref{equ2}) may blow up in finite time for super-critical mass and exist globally for sub-critical mass. For the super-critical case $m>2-\frac{2}{d}$, the existence and uniqueness of the solution to
   (\ref{equ2}) holds for any initial data; for subcritical case $m<2-\frac{2}{d}$, one observers either global existence
   or blow-up for different initial data. Another critical exponent $m=\frac{2d}{d+2}$, Chen et~al. \cite{chen2012multidimensional,chen2014exact} used  the properties of the free energy and stationary solutions to obtain an exact criterion for global existence and blow-up of solutions.

 The macroscopic Patlak–Keller–Segel model is based on the limiting behavior of the microscopic model see \cite{stevens2000derivation}. Nevertheless, in the process of deriving the macroscopic equations mentioned above, variations from the average value are overlooked. Additionally, in natural systems, random disturbances and  environmental noise are unavoidable. Meanwhile,  random disturbances and the nonlinearity term can lead to alterations in dynamic behavior. Therefore, the model needs to incorporate random spatio-temporal influences. Introducing a Gaussian random field to the system restores its natural randomness. This makes us to explore the traditional Patlak–Keller–Segel system influenced by the random term.

   Introducing a random term (or noise) into a partial differential equation causes a variety of new phenomena and can significantly affect the behavior of solutions.  Let us consider the stochastic aggregation-diﬀusion equation of
   Keller–Segel type
   \begin{align}\label{equ3}
   	d\rho  = [\Delta {\rho } - \chi \nabla  \cdot (\rho \nabla \Phi*\rho)]dt+\sigma(\rho )dW(t), ~x\in\mathbb{R}^d,t\ge 0.
   \end{align}
   For $d=2,3$, Huang and Qiu in \cite{huang2021microscopic} studied  the equation (\ref{equ3}) with the Bessel potential and $\sigma(\rho)=\rho$. The authors derived the stochastic partial differential equation (SPDE) as the large particle limit of a system  with shared noise (in divergence form), and  use a contraction argument similar to that in \cite{coghi2016propagation} to demonstrate the local existence and uniqueness of weak solutions over the interval $[0, T]$, with $T$  depending on the $L^4$-norm of the initial data. Additionally, they established that if the $L^4$-norm of the initial data is sufficiently small, the equation has a global weak solution. In \cite{2022On}, Misiats and  Stanzhytskyi proved  that the different types of random disturbances influence the properties of solutions to the stochastic Keller–Segel equation (\ref{equ3}) with $\sigma(\rho)=\nabla\rho $ or $\sigma(\rho)=\rho $. If the noise is in  divergence form, the equation has a
   global weak solution for small initial data.  Furthermore, if the noise is not in  divergence form, the authors showed that the solution has a ﬁnite time
   blow-up (with nonzero probability) for any nonzero initial data.  In \cite{2021The}, Hausenblas, Mukherjee and Tran demonstrated that a martingale solution exists for the one-dimensional stochastic version of the Patlak–Keller–Segel system with the noise interpreted in the Stratonvich sense.

   The stochastic porous media equation has been recently studied under various assumptions regarding the drift, including both additive and multiplicative noise. Consider the stochastic porous media equation \begin{eqnarray}\label{equ4}
   	\begin{cases}	d\rho-\Delta\Psi (\rho )dt  =\rho dW(t),&(0,T)\times\mathcal{O}, \\	\rho (x,0) = \rho_{0},&t= 0\times \mathcal{O} .\end{cases} 
   \end{eqnarray}Equation (\ref{equ4}) was studied under different assumptions on the operator $\Psi:\mathbb{R}\rightarrow\mathbb{R}$. In \cite{BARBU2007Existence}, the authors proved that if $\Psi$ is monotone and has polynomial growth, the stochastic porous media equation has a unique nonnegative solution for nonnegative initial data in $H_{2}^{-1}( \mathcal{O})$,  the existence results for equation (\ref{equ4}) were also derived in \cite{Viorel2006Weak,2002A,2008Stochastic,Jiagang2007Stochastic}. For more general nonlinear conditions, not assuming continuity of the drift
   or any growth condition at infinity, the reference \cite{2007Existence} showed the existence and uniqueness of strong solutions of the stochastic porous media equations. For the special case $\Psi (\rho )=\rho|\rho|^{m-1}$, $m\in (1,\infty )$, the existence and uniqueness of solutions to (\ref{equ4}) were proved in \cite{2018Nonlinear,2020ERGODICITY,2019Well}. For stochastic porous media equation driven by an infinite dimensional Wiener process, the variational approach, which has been established by Pardoux, Krylov and Rozovskii \cite{krylov2007stochastic,pardoux1975equation}, has been used intensively to prove the  existence and uniqueness of solutions. The main criterion for this approach to work is that the coefﬁcients satisfy certain monotonicity assumptions.
   
   But very little work has been done on the stochastic Keller-Segel  model with degenerate diffusion term. Recently in \cite{2022Martingale},  Hausenblas, Mukherjee, and Zakaria proved that the  stochastic chemotaxis system perturbed by a pair of Wiener processes exists a martingale solution on a two-dimensional domain and the integral is interpreted in the Stratonovich sense. Here, both the cell density and the concentration of the chemoattractant are influenced by time-homogeneous spatial Wiener processes. 
   
   With respect to the situations considered above, in the present work we study the stochastic degenerate  Keller-Segel system perturbed by the multiplicative It\^{o} noise on two and three dimensional domain. To our knowledge, this particular case has not been explored previously. Because of the nonlinear characteristics of the diffusion term in (\ref{equ}), the semigroup approach cannot be used for the stochastic equation. As a result, the traditional method for showing the existence and uniqueness of solutions is not applicable in this case. Due to a lack of coercivity of the nonlinear chemotactic term, it is impossiblet to apply the variational approach directly.  
   
   In this paper, according to the method of \cite{2022Martingale,liu2015stochastic},  we establish the global existence of martingale solutions for the equation (\ref{equ}), in which the interaction kernel $\Phi(x)$ can represent the Newtonian or Bessel  potential. We formulate a solution operator, analyzing the solution operator is well-defined under the variational approach and confirming solution operator's continuity and compactness within a suitably selected Banach space. Furthermore, we use a stochastic version of the Schauder fixed point theorem similar to that in \cite{2022Martingale},  which is speciﬁc to our problem in order to find  solutions on the interval $[0, T]$, for $T > 0$. In the process of proving the existence of solutions, we employ compactness arguments, which results in the original probability space being lost. Consequently, this implies that the solution will only be considered a weak solution in a probabilistic sense.
  
   In this work, we present results on global existence of  solutions when the diffusion exponent $m\ge 3$. Compared with the classical degenerate Patlak-Keller-Segel equation without noise term, the diffusion exponent is different. Because the   stochastic aggregation-diffusion equation of Keller-Segel type in two-dimensional space has a blow-up with nonzero probability for any nontrivial initial conditions, which is different from the classical Keller-Segel model with linear diffusion. Thus, we guess that the  appearance of noise term may be more beneficial for blow-up of solutions. Thus, choosing $m\ge 3$ is reasonable.

Let us provide an overview of the structure of this paper:
\begin{itemize}
    \item In Section 2, we introduce fundamental definitions related to linear operators and function spaces. Additionally, we present the classical variational approach for solving SPDEs in infinite-dimensional spaces.
    \item In Section 3, we establish the existence of  global weak solutions to equation (\ref{equ}) in bounded domains. First, we construct an integral operator on an appropriate Banach space and demonstrate its continuity and compactness. Then, we formulate a stochastic version of the Schauder fixed point theorem to prove the existence of martingale solutions.
\end{itemize}

 \section{Preliminaries}
In this section, we introduce fundamental definitions related to function spaces and linear operator. Additionally, we present the variational approach for solving stochastic partial differential equations (SPDEs) in infinite-dimensional spaces. 

 \subsection{Gelfand triple}
First, we recall the Schwartz space, which serves as a fundamental tool for defining the Fourier transform of functions. Let $C^{\infty } (\mathbb{R}^d )$ denote the set of smooth functions. Define the Schwartz space
 \begin{align*}
  \mathcal{S}(\mathbb{R}^{d}):=\left\{\phi\in C^{\infty } (\mathbb{R}^d ):\|\phi\|_{k,\mathcal{S}}<\infty,~~\forall k\in \mathbb{N}\right\},
 \end{align*}
 the seminorms is defined by
 \begin{align*}
  \|\phi\|_{k,\mathcal{S}}=\sup_{|\alpha| \leq k,x\in\mathbb{R}^d } (1+|x|)^{k} |\partial^{\alpha }\phi (x) |, 
 \end{align*}
 where  $\alpha\in \mathcal{A}:=\left\{\alpha=(\alpha_1,\cdot  \cdot \cdot ,\alpha_d)\right\}$, which is the set of all multi-indices and $|\alpha|=\alpha_1+\cdot  \cdot \cdot +\alpha_d$, $x\in \mathbb{R}^d $ and $\partial^{\alpha }\phi=\partial^{\alpha_1 }_{x_1}\cdot  \cdot \cdot \partial^{\alpha_d }_{x_d}\phi$.

Equipped with the family of seminorms $(|\phi|_{k,\mathcal{S}})_{k\in\mathbb{N}}$, the space $\mathcal{S}(\mathbb{R}^{d})$ forms a Fr$\acute{e}$chet space. Furthermore, we introduce the Schwartz tempered distribution space $\mathcal{S}'(\mathbb{R}^{d})$, where a tempered distribution on $\mathbb{R}^{d}$ is any continuous linear functional on $\mathcal{S}(\mathbb{R}^{d})$. Let $\mathcal{D}(\mathbb{R}^{d})$ denote the set of smooth functions with compact support in $\mathbb{R}^d$. Then, $\mathcal{D}(\mathbb{R}^{d})$ is dense in $\mathcal{S}(\mathbb{R}^{d})$. The Fourier transform $\mathcal{F}$ is a continuous linear operator from $\mathcal{S}(\mathbb{R}^{d})$ to itself. Both $\mathcal{F}$ and its inverse are defined as
    	\begin{align*}
  	  \mathcal{F}[f](\xi):=\frac{1}{(2 \pi)^{d / 2}} \int e^{-i \xi \cdot x} f(x)dx~ ,~\mathcal{F}^{-1}[f](x):=\frac{1}{(2 \pi)^{d / 2}} \int e^{i \xi \cdot x} f(x) d\xi.
  	\end{align*}
 for any $f\in \mathcal{S}(\mathbb{R}^{d})$. The Fourier transform  can be extended $\mathcal{S}'(\mathbb{R}^{d})$ in the sense that for $u\in \mathcal{S}'(\mathbb{R}^{d})$ it holds
 \begin{align*}\langle\mathcal{F}[u], f\rangle=\langle u, \mathcal{F}[f]\rangle, \forall f\in \mathcal{S}(\mathbb{R}^{d}).
  \end{align*}

  The Bessel potential for each  $ s\in \mathbb{R}$  is denoted by \begin{align*}
      J^{s}(u):=(1-\Delta)^{s / 2} u=\mathcal{F}^{-1}\left[\left(1+|\xi|^{2}\right)^{s / 2} \mathcal{F}[u]\right], 
  \end{align*} for  $u\in \mathcal{S}^{\prime}(\mathbb{R}^{d})$. The Bessel potential space $H_{p}^{s}(\mathbb{R}^{d}) $ for $  p \in(1, \infty)  $ and $ s \in \mathbb{R} $ is defined by
  \begin{align*}
  	H_{p}^{s}(\mathbb{R}^{d}):=\left\{u \in \mathcal{S}^{\prime}(\mathbb{R}^{d}):(1-\Delta)^{s / 2} u \in L^{p}(\mathbb{R}^{d})\right\},
  \end{align*}
  with the norm
\begin{align*}\|u\|_{H_{p}^{s}\left(\mathbb{R}^{d}\right)}:=\left\|(1-\Delta)^{s / 2} u\right\|_{L^{p}(\mathbb{R}^{d})}.
  \end{align*}
  Presently, $H_{p}^{s}(\mathcal{O})$ is the restriction of $ H_{p}^{s}(\mathbb{R}^{d})$ to $\mathcal{O}$ defined by
  \begin{align*}
  	H_{p}^{s}(\mathcal{O}):=\left\{f:f=g|_{\mathcal{O}},g\in{H_{p}^{s}}(\mathbb{R}^{d});\|f\|_{H_{p}^{s}(\mathcal{O})}=\operatorname*{inf}\|g\|_{H_{p}^{s}(\mathbb{R}^{d})}\right\},
  \end{align*}
  and, 
  $ H_0^{{1},{p}}(\mathcal{O}):=$completion of $C_0^\infty(\mathcal{O})$ with respect to $\|\cdot\|_{H_{p}^{1}(\mathcal{O})}.$

   For  $1<p<\infty, m \in \mathbb{N}$, the above Bessel potential spaces $ H_{p}^{m}$ can be characterized as Sobolev spaces
  \begin{align*}W^{m, p}(\mathbb{R}^{d}):=\Big\{f \in L^{p}(\mathbb{R}^{d}):\|f\|_{W^{m, p}\left(\mathbb{R}^{d}\right)}:=\sum_{\alpha \in \mathcal{A},|\alpha| \leq m}\left\|\partial^{\alpha} f\right\|_{L^{p}\left(\mathbb{R}^{d}\right)}<\infty\Big\}.\end{align*}

 Let $(U,\left\langle ,\right\rangle_{U} )$ and  $(H,\left\langle ,\right\rangle_{H} )$ be two separable Hilbert spaces. The space of
 all bounded linear operators from ${U}$ to $H$ is denoted by $L({U},H)$. A bounded linear operator 
 $	T: {U}\rightarrow H$
is called Hilbert–Schmidt operator if\begin{align*}
\sum_{k=1}^{\infty }\|Te_k\|_{H}^{2}\leq\infty,
\end{align*}where $\left \{ e _{k}  \right \}_{k\in \mathbb{N}}$ is an orthonormal basis of ${U}$. The space of all Hilbert–Schmidt operators from ${U}$ to ${H}$ is denoted by $L_2({U},H)$. Then the norm $\|\cdot\|_{L_2({U},H)}$ is represented by \begin{align*}
\|T\|_{L_2({U},{H})}^2:=\sum_{k=1}^{\infty }\|Te_k\|_{{H}}^{2}.
\end{align*}

 One of the main points in applying the variational approach is to find a suitable Gelfand triple, which is defined below. Let $V$ be a reﬂexive Banach space such that $V\subset H$ continuously and densely. Let $H^{*}$ be the dual space of $H$ and $V^{*}$ be the  dual space of $V$, it follows that $H^{*}\subset V^{*}$ continuously and densely. Identifying $H $ and $H^{*}$  via the Riesz isomorphism, we have that
 \begin{align*}
  	 	 		V\subset H\equiv  H^{*}\subset V^{*} 
\end{align*}
  	 	 	continuously and densely. If  $_{V^{*}}\langle\cdot, \cdot\rangle_{V}$  denotes the dualization between  $V^{*}$  and $ V$, it follows that
\begin{align}\label{2.6}
  	 	 		_{V^{*}}\langle z,v \rangle_{V}=\left \langle z ,v \right \rangle_{H}, \forall
  	 	 		z\in H, v\in V, 
\end{align}
$(V,H,V^{*})$ is called a Gelfand triple.

 \subsection{Well-posedness for general SPDEs} 
In this section, we present some general results on the existence and uniqueness of solutions to stochastic partial differential equations (SPDEs) based on a given Gelfand triple $(V,H,V^{*})$.

Let $(\Omega ,\mathcal{F},\mathbb{P}  )$ be a complete probability space with  a right-continuous complete filtration $\left \{\mathcal{F}_{t} \right \} _{t\ge 0}$. Assume $W(t)$, $t\in[0,T]$ is a cylindrical Wiener process in a separable Hilbert space $U$. We consider the following stochastic different equation on $H$. 
\begin{eqnarray}\label{eqn4}
	d\rho(t)  = A(t,\rho,\omega) dt + B(t,\rho,\omega)dW(t),
\end{eqnarray}
where the operator 
\begin{align*}
A:[0,\infty)\times V \times\Omega\to V^{*},~B:[0,\infty)\times V \times\Omega\to L_2(U,H)
\end{align*}
are progressively measurable. According to the reference \cite{liu2015stochastic}, we impose the following conditions on $A$ and $B$.

Suppose that there exist constants $\alpha\in(1,\infty),~\beta\in[0,\infty),~\theta\in (0,\infty),C_0\in\mathbb{R}$ and
a nonnegative adapted process $f\in L^{1}([0,T]\times\Omega,dt\otimes \mathbb{P})$ such that the following
conditions hold for all $u,v,w\in V$ and $(t,\omega)\in [0,T]\times\Omega.$

$(H1)$ (Hemicontinuity) The map
$\lambda \mapsto _{V^{*}}\langle A(t,u+\lambda v),w\rangle_{V}$ is continuous on $\mathbb{R}$.

$(H2)$ (Local monotonicity)
\begin{align*}
    &2_{V^{*}}\langle A(t,u)-A(t,v), u-v \rangle_{V}+\|B(t,u)-B(t,v)\|_{L_2(U,H)}^{2}\leq( f(t)+h(v))\|u-v\|_{H}^2,
\end{align*}
where $h:V\to [0,\infty)$ is a measurable hemicontinuous function and
locally bounded in V.

$(H3)$ (Coercivity)
  \begin{align*}
    &2_{V^{*}}\langle A(t,v), u\rangle_{V}+\|B(t,v)\|_{L_2(U,H)}^{2}\leq C_0\|v\|_{H}^2-\theta\|v\|_{V}^\alpha +f(t).
\end{align*}

$(H4)$ (Growth)
\begin{align*}
    \|A(t,v)\|_{V^{*}}^\frac{\alpha}{\alpha-1}\leq(f(t)+C_0\|v\|_{V}^\alpha)(1+\|v\|_{H}^\beta).
\end{align*}

Based on these conditions, in \cite{liu2015stochastic} the authors proved the existence of solutions on the equation (\ref{eqn4}), the result is as follows. 
\begin{lem} [\cite{liu2015stochastic}, Theorem 5.1.3]\label{lem5.1.3}
   Suppose $(H1)-(H4)$ hold for some $f\in L^{p/2}([0,T]\times \Omega,dt\otimes \mathbb{P})$ with $p\ge\beta+2$, and there exists a constant C such that
   \begin{align*}
       &\|B(t,v)\|_{L_2(U,H)}^{2}\leq C(f(t)+\|v\|_{H}^2),~t\in[0,T],v\in V;\\&
       h(v)\leq C(1+\|v\|_{V}^\alpha)(1+\|v\|_{H}^\beta),~v\in V.
   \end{align*}
   Then for any $\rho_0\in L^{p}(\Omega ,\mathcal{F}_{0} ,H)$, the equation (\ref{eqn4}) has a unique solution $\rho(t)$ such that for $t\in[0,T]$
   \begin{align*}
   \rho \left ( t \right ) \in L^{\alpha} \left ( \left [ 0,T \right ]\times\Omega, dt\otimes \mathbb{P}; V \right )\cap L^{2}\left (\left [ 0,T \right ]\times\Omega, dt\otimes \mathbb{P}; H  \right ), 
\end{align*}
where $\alpha$ is consistent with that in $(H3)$ and $\rho(t)$ satisifies \begin{align*}
       \mathbb{E}\Big[\sup _{0 \leq t \leq T}\|\rho \|_{H}^{p}\Big]<\infty.
   \end{align*}
\end{lem}
\subsection{Our setting}
In this paper, we set $U:=L^2(\mathcal{O})$, $V:=L^{m+1}(\mathcal{O})$ for any $m\geq 1$.  Moreover we define $H:=H_2^{-1}(\mathcal{O})$, then its  dual space $H^*=H_{0}^{1,2}(\mathcal{O})$.
Since $H_{0}^{1,2}(\mathcal{O}) \subset L^{2}(\mathcal{O}) $ continuously and densely, so is, $H_{0}^{1,2}(\mathcal{O}) \subset L^{2}(\mathcal{O})\subset L^{\frac{m+1}{m}}(\mathcal{O}) $. This implies that $$V=L^{m+1}(\mathcal{O})=(L^{\frac{m+1}{m}}(\mathcal{O}))^{*}\subset ( H_{0}^{1,2}(\mathcal{O}))^{*}\equiv H_2^{-1}(\mathcal{O}). $$ 
Namely, we have $V\subset H$. 
  Consequently, we get a Gelfand triple  with 
  \begin{equation}\label{2.13}
      V=L^{m+1}(\mathcal{O}) \subset H_2^{-1}(\mathcal{O})\equiv H_{0}^{1,2}(\mathcal{O})\subset( L^{m+1}(\mathcal{O}))^*=V^*.
  \end{equation}

Let us recall the following lemma to identify $H$ and its dual $H^*$ via the corresponding Riesz isomorphism $R:H\to H^{*}$ defined by $Rx:=<x,\cdot>_H,~x\in H.$ 
\begin{lem} 	  
[\cite{liu2015stochastic}, Lemma 4.1.12] The map $\Delta: H_{0}^{1,2}(\mathcal{O}) \rightarrow H_2^{-1}(\mathcal{O})$ is an isometric isomorphism. In particular,
$$\langle\Delta u, \Delta v\rangle_{H_2^{-1}(\mathcal{O})}=\langle u, v\rangle_{H_{0}^{1,2}},~\forall~  u, v \in H_{0}^{1,2}(\mathcal{O}). $$
    Furthermore,  $(-\Delta)^{-1}: H_2^{-1}(\mathcal{O}) \rightarrow H_{0}^{1,2}(\mathcal{O})$  is the Riesz isomorphism for $ H_2^{-1}(\mathcal{O})$, i.e. for every  $x \in H_2^{-1}(\mathcal{O})$,
\begin{align}\label{eq:risesz}
  	 \langle x, \cdot\rangle_{H_2^{-1}(\mathcal{O})}={ }_{H_{0}^{1,2}}\left\langle(-\Delta)^{-1} x, \cdot\right\rangle_{H_2^{-1}(\mathcal{O})}.
\end{align} 
\end{lem}
According to the above lemma, we can identify $H=H_2^{-1}(\mathcal{O})$ with its dual $H^*=H_{0}^{1,2}(\mathcal{O})$ by Risesz map $(-\Delta)^{-1}: H_2^{-1}(\mathcal{O}) \rightarrow H_{0}^{1,2}(\mathcal{O})$. This gives rise to a different dualization between $V$ and $V^*$. In particular, according to \eqref{2.6}, we have for any $u\in H$ and $\omega \in V$ 
\begin{equation}\label{3.1}
_{V^*}\langle u,\omega \rangle_V=\langle u,\omega \rangle_H= _{H^*}\langle (-\Delta)^{-1}u,\omega \rangle_H =\int_{\mathcal{O}}((-\Delta)^{-1}u)(x)w(x)dx\,
\end{equation}
where in the second equality we have used \eqref{eq:risesz}.

Now, we can rewrite equation (\ref{equ}) as \begin{eqnarray}
	d\rho(t)  = A(t,\rho,\omega) dt + B(t,\rho,\omega)dW(t),
\end{eqnarray}
where the operators
$A(\rho)$ and $B(\rho)$ are deﬁned by
\begin{align*}    
A(\rho):=A(t,\rho,\omega)=[\Delta {(|\rho| ^{m-1}\rho)} - \chi \nabla  \cdot (\rho \nabla\Phi*\rho)],~~ B(\rho):=B(t,\rho,\omega)=\sigma\rho.
\end{align*} 
Next we need to check that for $\rho\in L^{m+1}(\mathcal{O})$, $B$ is well deﬁned from $L^{m+1}(\mathcal{O})$ into the Hilbert-Schmidtis 
space $L_2(L^{2}(\mathcal{O}),H_{2}^{-1}(\mathcal{O}))$. 
Indeed, by the Sobolev embedding with $d\leq3$,
there is a constant $C>0$ such that
\begin{align*}
\|e_{k}\|_{L^{\infty}(\mathcal{O})}\leq C\|e_{k}\|_{H_{2}^{2}(\mathcal{O})}\leq C\|\Delta e_{k}\|_{L^{2}(\mathcal{O})}\leq C\lambda_{k}.
\end{align*}
Then, we get by elementary computations that
\begin{align*}
\|\rho e_{k}\|_{H_{2}^{-1}(\mathcal{O})}^2\leq C^2\lambda_{k}^2\|\rho \|_{H_{2}^{-1}(\mathcal{O})}^2.
\end{align*}
Defining
\begin{align*}
(B(\rho),\varphi):=\sum_{k=1}^{\infty }\sigma\mu_k\left\langle e_{k},\varphi \right\rangle_{ L^{2}(\mathcal{O})}^{2}\rho{e}_{k}, \rho\in L^{m+1}(\mathcal{O}),\varphi\in L^{2}(\mathcal{O}),
\end{align*}
where $\left\{\mu_k\right\}_k$ is a sequence of positive numbers. Throughout this paper we shall assume that \begin{align*}
    \sum_{k=1}^{\infty }{\mu_k}^2\lambda_{k}^2:=C_1<\infty
\end{align*}
For $\rho \in L^{m+1}(\mathcal{O})\subset H_{2}^{-1}(\mathcal{O}) $,  there exists a constant $C>0$ such that 
\begin{align}\label{normphi}
	\|B(\rho)\|_{L_2(L^{2}(\mathcal{O}),H_{2}^{-1}(\mathcal{O}))}^2=\sum_{k=1}^{\infty } \|\sigma \mu_k  \rho  e_{k}\|_{H_{2}^{-1}(\mathcal{O})}^2\leq C_1\| \rho \|_{H_{2}^{-1}(\mathcal{O})}^2\leq C\|\rho\|^2_{L^{m+1}(\mathcal{O})},
\end{align}
 i.e., the operator $B$ is a map given by
$$B:[0,\infty)\times L^{m+1}(\mathcal{O}) \times\Omega\to L_2(L^{2}(\mathcal{O}),H_{2}^{-1}(\mathcal{O})).$$
Therefore, we define
\begin{align*}
B(\rho):L^{2}(\mathcal{O})\rightarrow H_{2}^{-1}(\mathcal{O}),
\end{align*}
and \begin{align}
B(\rho)dW(t)=\sum_{k=1}^{\infty } \sigma\mu _{k}   e_{k} \rho (t)d\beta _{k}\left( t\right).
\end{align}
Furthermore, by the definition of $B$ and \eqref{normphi}, we have that $B$ is linear and Lipschitz continuous.

For the operator $A(\rho)$, due to the lack of coercivity in the nonlinear chemotactic term $\nabla\cdot(\rho \nabla\Phi \ast \rho)$, it is challenging to directly verify that 
\[
A: [0,\infty) \times L^{m+1}(\mathcal{O}) \times \Omega \to (L^{m+1}(\mathcal{O}))^{*}.
\]
To overcome this difficulty, later we will construct a solution operator and apply the Schauder fixed point theorem within the variational framework to establish the existence of solutions.

 \section{Global Existence of weak solutions to SPDE (\ref{equ}) }

\subsection{Preparatory work}
   The chemotactic term $\nabla\cdot(\rho \nabla\Phi \ast \rho)$ of the model 
  (\ref{equ}) arises as a gradient flow of a potential $\Phi$.  
  The convolution is to be understand  in $\mathcal{O}$, which be defined as
  \begin{align*}
      \Phi* \rho (x):= \int_{\mathcal{O}}\Phi(x-y)\rho (y) dy,~~x\in \mathcal{O}.
  \end{align*}
  Here, we consider the interaction  potential $\Phi$ is Newtonian or Bessel potential. By direct calculations, we get the following lemma.
\begin{lem}\label{lem2.4}
     Let  $\Phi$ be the interaction potential. Then the following holds.
     
     \begin{itemize}
         \item[(i)]  Let $\Phi(x)$ be the Newtonian  potential,  then $$\|\nabla \Phi\ast\rho\|_{L^{\infty}(\mathcal{O})}  \leq C\|\rho\|_{L^{m+1}(\mathcal{O})},~\text{for }~ m+1\ge d;$$
   \item[(ii)] Let $\Phi(x)$ be the Bessel potential,  then 
   \begin{align*}
   \|\nabla \Phi\ast\rho\|_{L^{\infty}(\mathcal{O})}  \leq C\|\rho\|_{L^{m+1}(\mathcal{O})}
   ,~\text{for }~ m+1>2.
   \end{align*}
     \end{itemize} 
  \end{lem}
 \begin{proof}
    
 (i) If $\Phi(x)$ is the Newtonian  potential 
  \begin{align*}
 \Phi (x)= \begin{cases}
 \frac{1}{d(d-2)\alpha (d)}|x|^{2-d},~&d\ge 3, \\
-\frac{1}{2\pi } \ln|x|,&~d=2,
\end{cases}
  \end{align*}
  where $\alpha(d) $ is the volume of $d$-dimension unit ball, we have  the estimates 
  \begin{align}
      |\nabla \Phi(x)|\leq \frac{C}{|x|^{d-1}},~  |\Delta \Phi(x)|\leq \frac{C}{|x|^{d}},~(x\ne 0)
  \end{align}
  for some constant $C>0.$
  Furthermore, for $m+1\ge d$, using the weak Young inequality, it holds that
  \begin{align}\label{2.11}
  \|\nabla \Phi\ast\rho\|_{L^{\infty}(\mathcal{O})}\leq C\|\frac{1}{|x|^{d-1}}\|_{L_w^{\frac{d}{d-1}}(\mathcal{O})}\left\|\rho\right\|_{L^{d}(\mathcal{O})}\leq C\|\rho\|_{L^{d}(\mathcal{O})}\leq C\|\rho\|_{L^{m+1}(\mathcal{O})}.
  \end{align}
  
  (ii) If $\Phi(x)$ is the Bessel potential (see \cite{stein1970singular})
  \begin{align}
  \Phi(x)=\int_{0}^{\infty } \frac{1}{(4\pi t)^{\frac{d}{2} } } e^{-\frac{|x|^{2}}{4t} - t} dt,
  \end{align}
  then it satisfies $\left |\nabla\Phi(x)\right |\sim
  |x|^{-d+1}$ as $|x|\to 0.$
  Following \cite{li2023optimal}, we can directly get that $\Phi(x)$ satisﬁes \begin{align}\label{2.1}
  	\|\Phi(x)\|_{L^{p}(\mathbb{R}^{d})}<\infty, 1\leq p<\infty,
  \end{align}
  \begin{align}\label{2.3}
  	\|\nabla \Phi(x)\|_{L^{p}(\mathbb{R}^d)}<\infty, 1\leq p<2.
  \end{align}
  Based on the properties of (\ref{2.1}) and (\ref{2.3}) for the Bessel potential, we know that
 
    \begin{align}\label{2.9}
 \|\nabla \Phi\ast\rho\|_{L^{\infty}(\mathcal{O})}\leq\|\nabla \Phi\|_{L^{\frac{m+1}{m}}(\mathcal{O})}\|\rho\|_{L^{m+1}(\mathcal{O})}\leq C\|\rho\|_{L^{m+1}(\mathcal{O})},~2<m+1<\infty.
  \end{align}
   \end{proof}
   
Next, we introduce the following definitions of some Banach spaces.
 \begin{defn}\label{def1}
  For Banach space $\left(X,\|\cdot\|_{X}\right)$, for all $  q \geq 1$  and $ 0 \leq t<\tau \leq T $, let  $S_{\mathcal{F}}^{q}([t, \tau] ; X) $  be the space of  $X$-valued,  $\mathcal{F}_{t}$-adapted and continuous processes $ \left\{\xi_{s}, s \in[t, \tau]\right\}$ with the norm
  \begin{align*}
  	\|\xi\|_{S_{\mathcal{F}}^{q}([t, \tau] ; X)}:=\left\{\begin{array}{l}\left(\mathbb{E} \sup _{s \in[t, \tau]}\left\|\xi_{s}\right\|_{X}^{q}\right)^{1 / q}, \quad q \in[1, \infty) ;\\\operatorname{ess} \sup _{\omega \in \Omega} \sup _{s \in[t, \tau]}\left\|\xi_{s}\right\|_{X}, \quad q=\infty.\end{array}\right.
  	\end{align*}
  \end{defn}
   \begin{defn}\label{def2}
   	Let  $L_{\mathcal{F}}^{q}([t, \tau] ; X)$   be the space of $ X$-valued, predictable processes  $\left\{\xi_{s}, s \in[t, \tau]\right\}$ with the norm
   \begin{align*}
   	\|\xi\|_{L_{\mathcal{F}}^{q}([t, \tau] ; X)}:=\left\{\begin{array}{l}\left(\mathbb{E} \int_{t}^{\tau}\left\|\xi_{s}\right\|_{X}^{q} \mathrm{~d} s\right)^{1 / q}, \quad q \in[1, \infty); \\\operatorname{ess} \sup _{\omega \in \Omega} \sup _{s \in[t, \tau]}\left\|\xi_{s}\right\|_{X}, \quad q=\infty.\end{array}\right.	\end{align*}
 \end{defn}
 
Note that both $S_{\mathcal{F}}^{q}$ and $L_{\mathcal{F}}^{q}$ are Banach spaces. With this, we can now formulate the concept of martingale solutions to the SPDE (\ref{equ}).
  \begin{defn}\label{def3}(Martingale solutions)
  Let $T > 0$, a  martingale solution of the equation (\ref{equ}) is a triple\begin{align*}
 \left(  (\Omega ,\mathcal{F},\left \{\mathcal{F}_{t} \right \} _{t\ge 0}  ,\mathbb{P}  ), \rho,(W_t)_{t\ge0}\right) 
  \end{align*}such that
  \item[(i)]$\mathfrak{A}:=(\Omega ,\mathcal{F},\left \{\mathcal{F}_{t} \right \} _{t\ge 0}  ,\mathbb{P}  )$ is a probability space with a complete, right-continuous filtration;
   \item[(ii)]$W(t)$ is a cylindrical Wiener process on $L^{2}\left ( \mathcal{O}  \right )$  over the probability
  space $\mathfrak{A}$;
 \item[(iii)]$\rho(t,\omega)$ is $\mathcal{F}_{t}$-measurable process  such that $\mathbb{P}$-$a.s.$,
\begin{align*}
   \rho \left ( t \right ) \in L_{\mathcal{F}}^{m+1} \left ( \left [ 0,T \right ] ,L^{m+1}\left ( \mathcal{O}  \right )   \right )\cap S_{\mathcal{F}}^{2}\left ( \left [ 0,T \right ]  ,H_{2}^{-1}\left ( \mathcal{O}  \right )  \right ),
\end{align*}
and it satisfies equation (\ref{equ})  over the probability
  space $\mathfrak{A}$. 
 \end{defn} 

 \subsection{Existence of weak solutions}
 In this section, we establish the existence of martingale solutions to the equation (\ref{equ}). First, we construct an integral operator on a suitable function space and demonstrate its continuity and compactness. Then, we formulate a stochastic version of the Schauder fixed point theorem to prove the existence of martingale solutions. 
\begin{thm}\label{thm}
For $m \geq 3$, let $\rho_0 \in L^{2}(\Omega ,\mathcal{F}_{0} ,H_{2}^{-1}(\mathcal{O} ))$ satisfy the condition 
\[
\mathbb{E}[\|\rho_{0}\|_{L^{m+1}}^{m+1}]<\infty.
\]
Then, there exists a martingale solution 
\[
\left(  (\Omega ,\mathcal{F},\left \{\mathcal{F}_{t} \right \} _{t\ge 0}  ,\mathbb{P}), \rho,(W_t)_{t\ge0}\right)
\]
to the equation (\ref{equ}) in the sense of Definition \ref{def3}. Moreover, there exists a  constant $C>0$ such that
\[
\mathbb{E}\Big[\sup_{0\leq t\leq T}\|\rho(t)\|_{H_{2}^{-1}(\mathcal{O} )}^{2}+4\int_{0}^{T}\|\rho(t)\|_{L^{m+1}(\mathcal{O} )}^{m+1} ds\Big]\leq C,
\]
and
\[
\mathbb{E}\Big[\sup_{0\leq t\leq T}\|\rho(t)\|_{L^{m+1}(\mathcal{O} )}^{m+1}+m^2(m+1)\int_{0}^{T}\|\rho^{m-1}(t)\nabla\rho(t)\|_{L^{2}(\mathcal{O} )}^{2} dt\Big]\leq C.
\]
\end{thm}

\begin{proof}
Since the model (\ref{equ}) contains chemotactic and degenerate diffusion terms, which are both nonlinear, thus we prove  Theorem 3.1 using the Schauder fixed point theorem. 

Deﬁne  the spaces 
\begin{align*}
S_{T}:= \left \{ \xi =\xi (x,t,\omega) \in S_{\mathcal{F} }^{2} ([0,T];H_{2}^{-1}(\mathcal{O} ) ):\mathbb{E}\Big[\sup _{0 \leq t \leq T}\|\xi(t) \|_{{H_{2}^{-1}(\mathcal{O})}}^{2}\Big]<\infty \right\},
\end{align*}
    and its subspace
\begin{align*}
X_{T}:= \left \{\xi\in S_{T}:\mathbb{E}\Big[\sup_{0\leq t\leq T}\|\xi(t)\|_{H_{2}^{-1}(\mathcal{O} )}^{2}+4\int_{0}^{T}\|\xi(t)\|_{L^{m+1}(\mathcal{O} )}^{m+1} dt\Big]\leq R_{1};\right. \\	\left.\mathbb{E}\Big[\sup_{0\leq t\leq T}\|\xi(t)\|_{L^{m+1}(\mathcal{O} )}^{m+1}+m^2(m+1)\int_{0}^{T}\||\xi(t)|^{m-1} \nabla \xi\|^{2}_{L^{2}  }dt\Big]\leq R_{2}\right\}. 		  	
\end{align*}
Let us deﬁne the operator $ \mathcal{T} $ acting on $S_{T}$ as follows. For $\xi\in S_{T}$, let $ \rho:=\mathcal{T}\xi $  solve the following model
\begin{eqnarray}\label{eqn5}
d\rho  = \bar{A} (\rho) dt + B(\rho) dW(t),
\end{eqnarray}
where 
$$
\bar{A} (\rho):=\Delta {(|\rho| ^{m-1}\rho)} - \chi \nabla  \cdot (\xi \nabla \Phi*\xi),~~B(\rho):=\sigma \rho.
$$
Firstly, according to below Proposition~\ref{prop1}, the operator is well deﬁned on $ S_T$. Secondly, according to Proposition~\ref{Prop2}, the operator $\mathcal{T}$ maps $X_T$ into itself; Next, 
 Proposition~\ref{prop3} shows that the operator $\mathcal{T}$ is a continuity map from $X_T$ into $X_T$; Finally, Proposition~\ref{prop4} proves the operator $\mathcal{T}$ maps $X_T$ to a precompact set.
 
Hence, for any $\xi\in X_T$ there exists a $\rho:=\mathcal{T}\xi$ and $\rho$ is the unique solution to the model (\ref{eqn5}). In this way, we formulate Schauder fixed point theorem and show that there exists a solution $\rho\in X_T$   to the model (\ref{equ}).
\end{proof}
\begin{prop}\label{prop1}
For $m\ge3$, fixed $\xi\in X_{T}$ and for all $T>0$, there exists a  solution $\rho$ to the model (\ref{eqn5}) such that	
\begin{align*}
\mathbb{E}\Big[\sup_{0\leq t\leq T}\|\rho(t)\|_{H_{2}^{-1}}^{2}+\int_{0}^{T}\|\rho(t)\|_{L^{m+1}}^{m+1} ds\Big]<\infty.
\end{align*}
\end{prop}
\begin{proof}
The main process of this proof is to verify that $\bar{A}(\rho)$ and $ B(\rho)$ satisfy  hypothesis $(H1)-(H4)$, and use Lemma \ref{lem5.1.3} to get the existence and uniqueness of solutions to the model (\ref{eqn5}) with $m\ge3$. Let us consider the Gelfand triple (\ref{2.13})
\begin{align}
	V:=L^{m+1}(\mathcal{O})\subset H_2^{-1}(\mathcal{O})\equiv H_{0}^{1,2}(\mathcal{O})\subset( L^{m+1}(\mathcal{O}))^{*}=:V^{*}.
\end{align}
 Noticing 
\begin{align} 
    \nonumber\left\|\nabla \cdot (\xi \nabla \Phi*\xi)\right\|_{H_{2}^{-1}(\mathcal{O})}\leq & C\left\|\xi (\nabla \Phi*\xi)\right\|_{L^{2}(\mathcal{O})}\leq C\left\|\xi\right\|_{L^{2}(\mathcal{O})}\left\|\nabla \Phi*\xi\right\|_{L^{\infty}(\mathcal{O})},
\end{align}
applying  Lemma \ref{lem2.4} and the Sobolev embedding theorem $L^{m+1}(\mathcal{O})\subset L^{2}(\mathcal{O})$  for $m\ge 1$, we have
\begin{align*}
   \left\|\nabla \cdot (\xi \nabla \Phi*\xi)\right\|_{H_{2}^{-1}(\mathcal{O})} 
\leq C\left\|\xi\right\|_{L^{2}(\mathcal{O})}\left\|\xi\right\|_{L^{m+1}(\mathcal{O})}\leq C\left\|\xi\right\|_{L^{m+1}(\mathcal{O})}^{2}<\infty,
\end{align*}
which implies $$\nabla  \cdot (\xi \nabla \Phi*\xi)\in H_{2}^{-1}(\mathcal{O})\text{,\quad}\forall\xi \in L^{m+1}(\mathcal{O}). $$
Fixed $t\in[0,T]$, $\omega\in\Omega$, for any $w\in V$, by (\ref{2.6}) and the Cauchy-Schwarz inequality, we obtain that
\begin{align}\label{3.3}
       \nonumber\left| _{V^{*}}\langle \nabla \cdot (\xi  \nabla \Phi*\xi),w\rangle_{V}\right|&=\left| \langle \nabla \cdot (\xi  \nabla \Phi*\xi),w\rangle_{H_{2}^{-1}}\right|\leq\|\nabla \cdot (\xi  \nabla \Phi*\xi)\| _{H_{2}^{-1}}\|w\| _{H_{2}^{-1}}\\\nonumber
&\leq C\|\xi \nabla \Phi*\xi\| _{L^{2}}\|w\| _{H_{2}^{-1}}\leq C\|\xi\|_{L^{2}}\|\nabla \Phi*\xi\|_{L^{\infty}}\|w\| _{H_{2}^{-1}}\\
&\leq C\|\xi\|_{L^{m+1}}^{2}\|w\| _{H_{2}^{-1}}.
\end{align}
Next,  using (\ref{3.1}) and (\ref{3.3}), we derive that
\begin{align}\label{3.4}
\nonumber\left| _{V^{*}}\langle \bar{A}(\rho),w\rangle_{V}\right|
&=\left| -\int_{\mathcal{O}} |\rho(x)|^{m-1}\rho(x)w(x)dx-\chi_{V^{*}}\langle \nabla \cdot (\xi  \nabla \Phi*\xi),w\rangle_{V}\right| \\\nonumber
&\leq\|\rho\|_{L^{m+1}}^{m}\|w\|_{L^{m+1}}+\chi\|\xi\|_{L^{2}}\|\nabla \Phi*\xi\|_{L^{\infty}}\|w\| _{H_{2}^{-1}}\\
&\leq C(\|\rho\|_{L^{m+1}}^{m}+\|\xi\|_{L^{m+1}}^{2})\|w\| _{H_{2}^{-1}} .
\end{align}
 
 Now, we will verify $\bar{A}(\rho)$ and $ B(\rho)$ satisfy  hypothesis $(H1)-(H4)$ successively.

\textbf{Veriﬁcation of $(H1)$ (Hemicontinuity):} Let $\rho_{1},\rho_{2},w\in V, t\in[0,T],\omega\in\Omega$. We need to show that for $\lambda\in\mathbb{R},|\lambda|\leq1$, \begin{align*}
\lambda \mapsto _{V^{*}}\langle \bar{A}(\rho_1+\lambda \rho _{2}),w\rangle_{V}
\end{align*}
is continuous on $\mathbb{R}$. Indeed, by the definition of $\bar{A}(\rho)$, we know that
\begin{align*}
 |_{V^{*}}\langle \bar{A}(\rho_1+\lambda \rho _{2})-\bar{A}(\rho_1),w\rangle_{V}|\leq\|(\rho_1+\lambda\rho_2)^{m}-\rho_1^{m}\|_{V^*}\|w\|_V.
\end{align*}
Due to  $\lim_{\lambda  \to 0}\|(\rho_1+\lambda\rho_2)^{m}-\rho_1^{m}\|_{V^*}=0$, we get 
\begin{align*}
    \lim_{\lambda  \to 0}|_{V^{*}}\langle \bar{A}(\rho_1+\lambda \rho _{2})-\bar{A}(\rho_1),w\rangle_{V}|=0.
\end{align*}
Hence, $(H1)$ holds.

\textbf{Veriﬁcation of $(H2)$ (Local monotonicity).} Let $\rho_{1},\rho_{2}\in V,~ t\in[0,T],~\omega\in\Omega$. Exploiting the fact that $(|\rho_1|^{m-1}\rho_1- |\rho_2|^{m-1}\rho_2)(\rho_1- \rho_2)\ge |\rho_1- \rho_2|^{m+1}~(m\ge 1)$ , we obtain\begin{align}
	-\int_{\mathcal{O}} (|\rho_1|^{m-1}\rho_1- |\rho_2|^{m-1}\rho_2)(\rho_1- \rho_2)dx\leq -\int_{\mathcal{O}}|\rho_1- \rho_2|^{m+1}dx=-\|\rho_1- \rho_2\|_{L^{m+1}}^{m+1}.
\end{align}
Then, by  (\ref{3.1}) and (\ref{normphi}), we have
\begin{align*}
	 &2_{V^{*}}\langle \bar{A}(\rho_1)-\bar{A}(\rho_2),\rho_1-\rho_2
	 \rangle_{V}+\|B(\rho_{1})-B(\rho_2)\|_{L_2(L^2(\mathcal{O}),H_{2}^{-1}(\mathcal{O}))}^{2}\\
     \leq& -2\int_{\mathcal{O}}(|\rho_1|^{m-1}\rho_1- |\rho_2|^{m-1}\rho_2)(\rho_1- \rho_2)dx+\|\sigma\rho_{1}-\sigma\rho_2\|_{L_2(L^2(\mathcal{O}),H_{2}^{-1}(\mathcal{O}))}^{2}\\
     \leq& -2\|\rho_1- \rho_2\|_{L^{m+1}}^{m+1}+C\|\rho_1-\rho_2\|_{H_{2}^{-1}(\mathcal{O})}^{2}\\\leq& C\|\rho_1-\rho_2\|_{H_{2}^{-1}(\mathcal{O})}^{2}.
\end{align*}
Hence, $\bar{A}(\rho)$ and $B(\rho)$ satisfy $(H2) $ with $f(t)=C,~h(v)=0$.

\textbf{Veriﬁcation of $(H3)$ (Coercivity):} Let $\rho\in V,t\in[0,T],\omega\in\Omega$. Then, by (\ref{3.4}) and the Young inequality, we get
\begin{align*}
	 &2_{V^{*}}\langle \bar{A}(\rho),\rho \rangle_{V}+\|B(\rho)\|_{L_2(L^2(\mathcal{O}),H_{2}^{-1}(\mathcal{O}))}^{2}  \\= &-2\int_{\mathcal{O}} |\rho|^{m-1}\rho^2 dx-\chi_{V^{*}}\langle \nabla \cdot (\xi  \nabla \Phi*\xi),\rho\rangle_{V}+\|\sigma\rho\|_{L_2(L^2(\mathcal{O}),H_{2}^{-1}(\mathcal{O}))}^{2} \\\leq&-2\|\rho\|_{L^{m+1}}^{m+1}+\chi\|\xi\|_{L^{2}}\|\nabla \Phi*\xi\|_{L^{\infty}}\|\rho\|_{H_{2}^{-1}}+ C\|\rho\|_{H_{2}^{-1}}^{2}\\\leq&-2\|\rho\|_{L^{m+1}}^{m+1} +C\|\xi\|_{L^{m+1}}^{4} +C\|\rho\| _{H_{2}^{-1}}^{2}.
\end{align*}
Set $f(t):=C\|\xi(t)\|_{L^{m+1}}^{4}$, we are easy to check  $f\in L^{1}([0,T]\times \Omega,dt\otimes \mathbb{P} )$. In fact,   if $m\ge3$, then $\frac{4}{m+1}\leq1$, by the H\"{o}lder inequality in time and the Jessen inequality, for $\xi\in X_{T}$, we know \begin{align*}
	\mathbb{E} \Big[\int_{0}^{T} \|\xi(t)\|_{L^{m+1}}^{4}dt\Big]\leq \mathbb{E} \Big[T^\frac{m-3}{m+1}\Big(\int_{0}^{T} \|\xi(t)\|_{L^{m+1}}^{m+1} dt\Big)^{\frac{4}{m+1} }\Big]\leq C(T)\mathbb{E} \Big[\int_{0}^{T} \|\xi(t)\|_{L^{m+1}}^{m+1} dt\Big]^{\frac{4}{m+1} }.
\end{align*}
 Hence, $(H3)$ holds for $\theta=1$.

\textbf{Veriﬁcation of $(H4)$ (Growth):} Let $\rho\in V,t\in[0,T],\omega\in\Omega$.  Using the continuous embedding of $L^{m+1}(\mathcal{O})$ into $H_{2}^{-1}(\mathcal{O})$ and by (\ref{3.4}), we infer that \begin{align*}
\|\bar{A}(\rho)\|_{V^*}\leq C(\|\rho\|_{L^{m+1}}^{m}+\chi\|\xi\|_{L^{m+1}}^{2}).\end{align*}
Thus
\begin{align*}
\|\bar{A}(\rho)\|_{V^*}^\frac{m+1}{m}\leq C(\|\rho\|_{L^{m+1}}^{m+1}+\chi\|\xi\|_{L^{m+1}}^\frac{2(m+1)}{m}).
\end{align*}
Set $f(t):=\|\xi(t)\|_{L^{m+1}}^\frac{2(m+1)}{m}$, we  get
\begin{align*}
\mathbb{E} \Big[\int_{0}^{T} \|\xi(t)\|_{L^{m+1}}^{\frac{2(m+1)}{m} }dt\Big]\leq \mathbb{E} \Big[T^\frac{m-2}{m}\Big(\int_{0}^{T} \|\xi(t)\|_{L^{m+1}}^{m+1} dt\Big)^{\frac{2}{m} }\Big]\leq C(T) \mathbb{E} \Big[\int_{0}^{T} \|\xi(t)\|_{L^{m+1}}^{m+1} dt\Big]^{\frac{2}{m}}.
\end{align*}
Furthermore, by $\xi\in X_{T}$, we can infer that $f\in L^{1}([0,T]\times \Omega,dt\otimes \mathbb{P})$. Hence, $(H4)$ is satisfied if  $\alpha=m+1,~\beta=0$.

In the above discussion we have shown that the hypothesis $(H1)-(H4)$ are satisﬁed. Therefore, by an application of Lemma \ref{lem5.1.3}, for $\xi\in X_{T}$, the model (\ref{eqn5})  has a unique solution $\rho \left (x, t \right ) \in L_{\mathcal{F}}^{m+1} \left ( \left [ 0,T \right ] ;L^{m+1}\left ( \mathcal{O}  \right )   \right )\cap S_{\mathcal{F}}^{2}\left ( \left [ 0,T \right ];H_{2}^{-1}\left ( \mathcal{O}  \right )  \right )$. 
\end{proof}
\begin{prop}\label{Prop2}
 For all $\xi\in X_{T}$, let $\rho
$ be
the unique solution of the model (\ref{eqn5}) with $m\ge3$,  then there exist numbers $R_{1}>0$ and $R_{2}>0$ such that the operator $\mathcal{T}$ maps $X_{T}$ into itself.
\end{prop}
\begin{proof}
 We prove the proposition in two steps, by applying the It\^{o} formula to the function. In the Step 1,  we get an estimate on $\mathbb{E}\Big[\sup_{0\leq t\leq T}\|\rho(t)\|_{{H_{2}^{-1}}(\mathcal{O})}^{2}\Big]  $ depending on the norm of $\xi$, from which we get a condition for $R_1$. In the Step 2, we get an estimate on $\mathbb{E}\Big[\sup_{0\leq t\leq T}\|\rho(t)\|_{L^{m+1}(\mathcal{O})}^{m+1}\Big]  $ depending on the norm of $\xi$ and $R_1$, which gives
us a condition for $R_2$.

Step 1: Applying the It\^{o} formula (\cite{liu2015stochastic}, Theorem 4.2.5) to  $\|\rho(t) \|_{{H_{2}^{-1}}(\mathcal{O})}^{2}$, for $t\in[0,T]$, we have that
  \begin{align}\label{3.6}
\nonumber\|\rho(t) \|_{{H_{2}^{-1}}}^{2}=&\|\rho_0 \|_{{H_{2}^{-1}}}^{2}+2\int_{0}^{t}{_{V^{*}}\langle \Delta  (|\rho(s)|^{m-1}\rho(s)),\rho(s) \rangle_{V}-\chi _{V^{*}}\langle \nabla \cdot (\xi \nabla \Phi*\xi)(s),\rho(s) \rangle_{V}}ds\\
\nonumber&+\int_{0}^{t} \|\sigma\rho(s)\|_{L_2(L^2,H_{2}^{-1})}^{2}ds+2\int_{0}^{t}\langle \rho(s) ,\sigma \rho(s) dW(s) \rangle_{H_{2}^{-1}}\\
\nonumber\leq&\|\rho_0 \|_{{H_{2}^{-1}}}^{2}-2\int_{0}^{t}\|\rho (s)\|_{L^{m+1}}^{m+1}ds+2\chi\int_{0}^{t}\|\xi(s)\|_{L^{2}}\|\nabla \Phi(s)*\xi(s)\|_{L^{\infty}}\|\rho(s)\|_{H_{2}^{-1}}ds\\
&+\sum_{k=1}^{\infty } \int_{0}^{t}\sigma ^{2}\mu_k^2\|\rho (s)e_{k}\|_{{H_{2}^{-1}}}^{2}ds+2\sum_{k=1}^{\infty }\int_{0}^{t}\langle \rho(s) ,\sigma \rho(s)e_{k} \rangle_{H_{2}^{-1}}d\beta_{k}(s).
\end{align}
Taking supremum over $[0,T]$ and the expectation from both sides of (\ref{3.6}), we get \begin{align}\label{3.2.1}
&\nonumber\frac{1}{2} \mathbb{E}\Big[\sup _{0 \leq t \leq T}\|\rho(t)\|_{H_{2}^{-1}}^{2}\Big]-\frac{1}{2} \mathbb{E}\|\rho_{0}\|_{H_{2}^{-1}}^{2}+\mathbb{E}\Big[ \int_{0}^{T}\|\rho (t)\|_{L^{m+1}}^{m+1}dt\Big] \\\nonumber\leq&\chi\mathbb{E}\Big[ \int_{0}^{T}\|\xi(t)\|_{L^{2}}\|\nabla \Phi(t)*\xi(t)\|_{L^{\infty}}\|\rho(t)\|_{H_{2}^{-1}}dt\Big] +C\mathbb{E}\Big[\int_{0}^{T}\|\rho (t)\|_{{H_{2}^{-1}}}^{2}dt\Big]\\&+\mathbb{E}\Big[\sup _{0 \leq t \leq T} \sum_{k=1}^{\infty}\Big| \int_{0}^{t}\langle \rho(s) ,\sigma \rho(s)e_{k} \rangle_{H_{2}^{-1}}d\beta_{k}(s)\Big| \Big].
\end{align}
First,  applying Lemma \ref{lem2.4}  and the Young inequality  to estimate the first term on the right of (\ref{3.2.1}), we get
\begin{align}\label{3.2.3}
\nonumber&\mathbb{E}\Big[ \int_{0}^{T}\|\xi(t)\|_{L^{2}}\|\nabla \Phi(t)*\xi(t)\|_{L^{\infty}}\|\rho(t)\|_{H_{2}^{-1}}dt\Big]\\\nonumber\leq& C\mathbb{E}\Big[ \int_{0}^{T}\|\xi(t)\|_{L^{m+1}}^{2}\|\rho(t)\|_{H_{2}^{-1}}dt\Big]\\\nonumber\leq&\mathbb{E}\Big[\varepsilon \int_{0}^{T}\|\xi(t)\|_{L^{m+1}}^{4}+C\|\rho(t)\|_{H_{2}^{-1}}^{2}dt\Big]\\\leq& \varepsilon T^{\frac{m-3}{m+1}} \mathbb{E}\Big[ \int_{0}^{T}\|\xi(t)\|_{L^{m+1}}^{m+1}dt\Big] ^\frac{
4}{m+1}+C\mathbb{E}\Big[ \int_{0}^{T}\|\rho(t)\|_{H_{2}^{-1}}^{2}dt\Big].
\end{align}
Next, the last term on the right of (\ref{3.2.1}) can be estimated by the Burkholder-Davis-Gundy inequality (\cite{da2014stochastic}, Theorem 4.36) and the Young inequality 
\begin{align}\label{3.2.2}
&\nonumber\mathbb{E}\Big[\sup _{0 \leq t \leq T} \sum_{k=1}^{\infty}\Big| \int_{0}^{t}\langle \rho(s) ,\sigma \rho(s)e_{k} \rangle_{H_{2}^{-1}}d\beta_{k}(s)\Big| \Big]\\
\nonumber\leq& C\mathbb{E}\Big[ \Big( \int_{0}^{T}\|\rho(t)\|_{H_{2}^{-1}}^{2}
\|\sigma\rho(t)\|_{L_2(L^2,H_{2}^{-1})}^{2}dt\Big) ^\frac{1}{2}\Big]\\
\nonumber\leq& C\mathbb{E}\Big[\Big(\sup_{0 \leq t \leq T}\|\rho(t)\|_{H_{2}^{-1}}^2 \Big )^\frac{1}{2}\Big( \int_{0}^{T}\|\sigma \mu_{k}e_{k}\rho(t)\|_{H_{2}^{-1}}^{2}dt\Big)^ \frac{1}{2}\Big]\\
\leq& \frac{1}{4}\mathbb{E}\Big[\sup_{0 \leq t \leq T}\|\rho(t)\|_{H_{2}^{-1}}^{2} \Big] + C\mathbb{E}\Big[  \int_{0}^{T}\|\rho(t)\|_{H_{2}^{-1}}^{2}dt\Big].
\end{align}
Applying the estimates (\ref{3.2.3}) and   (\ref{3.2.2}) to  (\ref{3.2.1}), we get \begin{align*}
&\frac{1}{4} \mathbb{E}\Big[\sup _{0 \leq t \leq T}\|\rho(t)\|_{H_{2}^{-1}}^{2}\Big]+\mathbb{E}\Big[ \int_{0}^{T}\|\rho (t)\|_{L^{m+1}}^{m+1}dt\Big]\\\leq&\frac{1}{2} \mathbb{E}\|\rho_{0}\|_{H_{2}^{-1}}^{2}+\varepsilon T^{\frac{m-3}{m+1}} R_1^\frac{4}{m+1}+C\mathbb{E}\Big[ \int_{0}^{T}\|\rho(t)\|_{H_{2}^{-1}}^{2}dt\Big].
\end{align*}
Using  Gronwall lemma yields that
 \begin{align}\label{3.16}
\nonumber&\mathbb{E}\Big[\sup _{0 \leq t \leq T}\|\rho(t)\|_{H_{2}^{-1}}^{2}\Big]+4\mathbb{E}\Big[ \int_{0}^{T}\|\rho (t)\|_{L^{m+1}}^{m+1}dt\Big]\\\leq&\Big[2 \mathbb{E}\|\rho_{0}\|_{H_{2}^{-1}}^{2}+\varepsilon T^{\frac{m-3}{m+1}} R_{1}^\frac{4}{m+1}\Big] e^{CT}\leq C(T) \Big[ \mathbb{E}\|\rho_{0}\|_{H_{2}^{-1}}^{2}+\varepsilon T^{\frac{m-3}{m+1}}R_{1}^\frac{4}{m+1}\Big] \leq R_1,
\end{align}
where we choose $R_1$  large enough such that the last inequality holds.

Step 2:  Estimate $\|\rho(t) \|_{{L^{m+1}}(\mathcal{O})}^{m+1}$ by the It\^{o}  formula.

Applying the It\^{o}  formula (\cite{2008Stochastic}, Proposition 1.2.3) to  $\|\rho (t) \|_{{L^{m+1}}(\mathcal{O})}^{m+1}$
 and using standard calculation, we obtain 
 \begin{align}\label{3.9}
 	\nonumber\|\rho(t)\|_{L^{m+1} }^{m+1}=&\|\rho_{0}\|_{L^{m+1} }^{m+1} +(m+1)\sigma \sum_{k=1}^{\infty } \int_{0}^{t} \int_{\mathcal{O} }|\rho(s)|^{m-1}\rho^2(s) e_{k}dxd\beta _{k}(s) \\
    \nonumber& -(m+1)\chi\int_{0}^{t} \int_{\mathcal{O} }|\rho(s)|^{m-1}\rho(s) \nabla \cdot(\xi \nabla \Phi*\xi )dxds\\
    \nonumber&+(m+1)\int_{0}^{t} \int_{\mathcal{O} }|\rho(s)|^{m-1}\rho(s) \Delta (|\rho(s)|^{m-1}\rho(s)) dxds\\
    \nonumber&+\frac{\sigma ^{2} }{2} m(m+1)\sum_{ k=1}^{\infty} \int_{0}^{t} \int_{\mathcal{O} }|\rho(s)|^{m-1} |\rho(s) e_{k}|^{2}dxds \\ 
 	\nonumber\le& \|\rho_0\|_{L^{m+1} }^{m+1}+(m+1)\sigma \sum_{k=1}^{\infty }\int_{0}^{t} \int_{\mathcal{O} }|\rho(s)|^{m+1} e_{k}dxd\beta_{k}(s)\\
    \nonumber& +(m+1)\chi\int_{0}^{t} \int_{\mathcal{O} }\nabla(|\rho(s)|^{m-1} \rho(s)) \xi(\nabla \Phi*\xi )dxds\\
 	&-m^{2} (m+1)\int_{0}^{t} \int_{\mathcal{O} }\rho^{2m-2}(s) |\nabla \rho (s)|^{2} dxds+C(\sigma ,m)\int_{0}^{t} \|\rho(s)\|_{L^{m+1} }^{m+1}ds.
 \end{align}
 The third term on the right of (\ref{3.9}) is done with the H\"{o}lder inequality, we  get
 \begin{align}\label{3.10}
 \nonumber&(m+1)\chi\int_{\mathcal{O} } \nabla(|\rho(s)|^{m-1} \rho(s)) \xi(\nabla \Phi*\xi )dx\\
 \nonumber\le& m(m+1)\chi\||\rho(s)|^{m-1} \nabla \rho \|_{L^{2}  }\|\xi (\nabla \Phi*\xi )\|_{L^{2}  } \\ \le&\frac{m^2(m+1) }{2}\||\rho(s)|^{m-1} \nabla \rho \|^{2}_{L^{2}  }+C(\chi )\|\xi\|^{2} _{L^{2}}\|\nabla \Phi*\xi\|^{2} _{L^{\infty }}. 
 \end{align}
 Furthermore, putting (\ref{3.10}) into (\ref{3.9}),   we obtain 
 \begin{align}\label{3.113}
 \nonumber&\|\rho(t)\|_{L^{m+1}  }^{m+1} +\frac{m^2(m+1) }{2}\int_{0}^{t}\||\rho(s)|^{m-1} \nabla \rho \|^{2}_{L^{2}  }ds\\
 \nonumber\le& \|\rho_{0} \|_{L^{m+1}  }^{m+1}+C(\chi )\int_{0}^{t}\|\xi(s)\|^{2} _{L^{2} }\|\nabla \Phi*\xi\|^{2} _{L^{\infty }}ds+C(\sigma ,m)\int_{0}^{t}\|\rho(s) \|_{L^{m+1}  }^{m+1}ds\\
 &
 +(m+1)\sigma \sum_{k=1}^{\infty }\int_{0}^{t} \int_{\mathcal{O} }|\rho(s)|^{m+1} e_{k}dxd\beta_{k}(s).
 \end{align}
 Taking supremum over $[0,T]$ and the expectation for (\ref{3.113}), it is easy to see that
 \begin{align}\label{3.2.4}
 &\nonumber\mathbb{E}\Big[\sup_{0\le t\le T}\|\rho(t)\|_{L^{m+1} }^{m+1}\Big ]+\frac{m^2(m+1) }{2} \mathbb{E}\Big [\int_{0}^{T} \||\rho(t)|^{m-1} \nabla \rho \|^{2}_{L^{2}  }dt \Big]\\\nonumber
 	\le& \mathbb{E}\Big[\|\rho_{0} \|_{L^{m+1} }^{m+1}\Big]+C(\chi )\mathbb{E}\Big[\int_{0}^{T}\|\xi(t)\|^{2} _{L^{2} }\|\nabla \Phi*\xi\|^{2} _{L^{\infty }}dt\Big]+C(\sigma ,m)\mathbb{E}\Big[\int_{0}^{T}\|\rho(t) \|_{L^{m+1} }^{m+1}dt\Big]\\&+(m+1)\sigma \mathbb{E}\Big[\sup_{0\le t\le T}\sum_{k=1}^{\infty }\Big| \int_{0}^{t} \int_{\mathcal{O} }|\rho(s)|^{m+1} e_{k}dxd\beta_{k}(s)\Big|  \Big]. 
    \end{align}
 For the last term on the right of (\ref{3.2.4}),  using the Burkholder–Davis–Gundy inequality and the Young inequality, we get  
 \begin{align}\label{3.2.5}
 &\nonumber(m+1)\sigma \mathbb{E}\Big[\sup_{0\le t\le T}\sum_{k=1}^{\infty }\Big| \int_{0}^{t} \int_{\mathcal{O} }|\rho(s)|^{m+1} e_{k}dxd\beta_{k}(s)\Big|  \Big]\\\nonumber\le &C \mathbb{E}\Big[\Big( \int_{0}^{T} \Big( \int_{\mathcal{O} }|\rho(t)|^{m+1}dx\Big) ^{2} dt\Big) ^{\frac{1}{2} }  \Big]\\\nonumber
 	\le & C \mathbb{E}\Big[\Big( \int_{0}^{T} \|\rho(t)\|_{L^{m+1} }^{m+1}\|\rho(t)\|_{L^{m+1} }^{m+1}dt\Big) ^{\frac{1}{2} }  \Big]\\\nonumber\le& C \mathbb{E}\Big[\Big(\sup_{0\le t\le T}\|\rho(t)\|_{L^{m+1} }^{m+1}\Big)^{\frac{1}{2}}\Big( \int_{0}^{T} \|\rho(t)\|_{L^{m+1} }^{m+1}dt\Big)^{\frac{1}{2} }   \Big]\\\le& \frac{1}{2}  \mathbb{E}\Big[\sup_{0\le t\le T}\|\rho(t)\|_{L^{m+1} }^{m+1}\Big]+C\mathbb{E}\Big[\int_{0}^{T} \|\rho(t)\|_{L^{m+1} }^{m+1}dt  \Big].
 \end{align}
 Bring the estimate (\ref{3.2.5}) into (\ref{3.2.4}) and using  Lemma \ref{lem2.4}, we get
 \begin{align*}
 	&\frac{1}{2}  \mathbb{E}\Big[\sup_{0\le t\le T}\|\rho(t)\|_{L^{m+1} }^{m+1}\Big]+\frac{m^2(m+1) }{2} \mathbb{E}\Big [\int_{0}^{T} \||\rho(t)|^{m-1} \nabla \rho \|^{2}_{L^{2}  }dt \Big]\\\le&   \mathbb{E}\Big[\|\rho_0\|_{L^{m+1} }^{m+1}\Big]+C(\chi)\mathbb{E}\Big[\int_{0}^{T} \|\xi\|_{L^{m+1} }^{4} dt \Big]+C(\sigma ,m)\mathbb{E}\Big[\int_{0}^{T} \|\rho(s)\|_{L^{m+1} }^{m+1} dt \Big]\\\le&   \mathbb{E}\Big[\|\rho_0\|_{L^{m+1} }^{m+1}\Big]+C(\chi,T)R_{1}^\frac{4}{m+1}+C(\sigma ,m)\mathbb{E}\Big[\int_{0}^{T} \|\rho(s)\|_{L^{m+1} }^{m+1} dt \Big].
 \end{align*}
Using Gronwall lemma, we know that\begin{align}\label{3.17}
 \nonumber&\mathbb{E}\Big[\sup_{0\le t\le T}\|\rho(t)\|_{L^{m+1} }^{m+1}\Big]+m^2(m+1) \mathbb{E}\Big [\int_{0}^{T} \||\rho(t)|^{m-1} \nabla \rho \|^{2}_{L^{2}  }dt \Big]\\\le&   \Big[2\mathbb{E}\|\rho_{0} \|_{L^{m+1} }^{m+1}+C(\chi,T)R_{1}^\frac{4}{m+1}+C(\sigma ,m)\Big]e^{CT} \le R_{2},
 \end{align}
where we choose $R_2$  large enough such that the last inequality holds.
 
 Combining the inequalities (\ref{3.16}) and (\ref{3.17}), we have shown that there exist $R_1,~R_2 > 0$ such that $\mathcal{T}$ maps $X_{T}$ in itself. This completes the proof of Proposition 3.2.
\end{proof}
Next, we prove the continuity of the solution operator $\mathcal{T}$.  
\begin{prop}\label{prop3}
For all $\xi_1,\xi_2\in X_T$, let $\rho_1:=\mathcal{T}(\xi_1)$, $\rho_2:=\mathcal{T}(\xi_2)$ be the corresponding solutions of the model (\ref{eqn5}) with the same initial data $\rho_0$.  Then there exist some $\delta > 0$ and a constant $C > 0$ such that \begin{align}\label{3.33}
	\|\mathcal{T}(\xi_1)-\mathcal{T}(\xi_2)\|_{S_{T}}\leq C\|\xi_1-\xi_2\|_{S_{T}}^{\delta}.
\end{align}
\end{prop}
\begin{proof}
Since $\rho_1$ and $\rho_2$ are solutions of the model (\ref{eqn5}),  we can easily deduce that $\rho_1- \rho_2$ satisﬁes
\begin{align*}
	d(\rho_1- \rho_2) =& [\Delta( |\rho_1|^{m-1}\rho_1-|\rho_2|^{m-1}\rho_2) - \chi \nabla  \cdot (\xi_1( \nabla \Phi*\xi_1)-\xi_2( \nabla \Phi*\xi_2))]dt\\& + \sigma (\rho_1- \rho_2)dW(t).
\end{align*}
 Using It\^{o} formula to $\|(\rho_1- \rho_2) (t)\|_{{H_{2}^{-1}}(\mathcal{O})}^{2}$, we obtain\begin{align}\label{3.18}
\nonumber&\frac{1}{2}\| (\rho_1- \rho_2)(t) \|_{{H_{2}^{-1}}}^{2}\\
\nonumber=&\int_{0}^{t}{_{V^{*}}\langle \Delta( |\rho_1|^{m-1}\rho_1-|\rho_2|^{m-1}\rho_2),\rho_1- \rho_2 \rangle_{V}}ds\\
\nonumber&-\chi\int_{0}^{t}{ _{V^{*}}\langle \nabla  \cdot (\xi_1( \nabla \Phi*\xi_1)-\xi_2( \nabla \Phi*\xi_2)), \rho_1- \rho_2 \rangle_{V}}ds\\
\nonumber&+\frac{1}{2}\int_{0}^{t} \|\sigma(\rho_1- \rho_2) \|_{L_2(L^2,H_{2}^{-1})}^{2}ds+\int_{0}^{t}\langle \rho_1- \rho_2 ,\sigma (\rho_1- \rho_2)dW(s) \rangle_{H_{2}^{-1}}\\
\nonumber\leq&-\int_{0}^{t}\| \rho_1- \rho_2\|_{L^{m+1}}^{m+1}ds+\frac{1}{2}\sum_{k=1}^{\infty } \int_{0}^{t}\sigma ^{2}\mu_k^2\|(\rho_1- \rho_2) e_{k}\|_{{H_{2}^{-1}}}^{2}ds\\
\nonumber&-\chi\int_{0}^{t}{ _{V^{*}}\langle \nabla  \cdot (\xi_1( \nabla \Phi*\xi_1)-\xi_2( \nabla \Phi*\xi_2)), \rho_1- \rho_2\rangle_{V}}ds\\&+\sum_{k=1}^{\infty }\int_{0}^{t}\langle \rho_1- \rho_2 ,\sigma (\rho_1- \rho_2)e_{k} \rangle_{H_{2}^{-1}}d\beta_{k}(s).
\end{align}
Taking supremum over $ [0, T]$ and the expectation from both sides of (\ref{3.18}), we obtain
\begin{align*}
&\frac{1}{2} \mathbb{E}\Big[\sup _{0 \leq t \leq T}\|(\rho_1- \rho_2)(t)\|_{H_{2}^{-1}}^{2}\Big]+\mathbb{E}\Big[ \int_{0}^{T}\|(\rho_1- \rho_2)(t)\|_{L^{m+1}}^{m+1}dt\Big] \\\leq&\chi\mathbb{E}\Big[ \int_{0}^{T}{\left| _{V^{*}}\left\langle  \nabla  \cdot (\xi_1( \nabla \Phi*\xi_1)-\xi_2( \nabla \Phi*\xi_2))(t), (\rho_1- \rho_2)(t) \right\rangle _{V}\right|dt }\Big] \\&+\mathbb{E}\Big[\sup _{0 \leq t \leq T} \sum_{k=1}^{\infty}\Big| \int_{0}^{t}\langle (\rho_1- \rho_2)(s) ,\sigma (\rho_1- \rho_2)(s)e_{k} \rangle_{H_{2}^{-1}}d\beta_{k}(s)\Big| \Big]\\&+C\mathbb{E}\Big[\int_{0}^{T}\|(\rho_1- \rho_2) (t)\|_{{H_{2}^{-1}}}^{2}\Big]dt\\=:&J_1(t)+J_2(t)+J_3(t).
\end{align*}

Now we deal with $J_1(t)$. First, we divide the term into the following sum
\begin{align*}
J_1(t)\leq&\chi\mathbb{E}\Big[ \int_{0}^{T}{| _{V^{*}}\langle \nabla  \cdot ((\xi_1-\xi_2) (\nabla \Phi*\xi_1)), \rho_1- \rho_2\rangle _{V}|dt }\Big]\\&+\chi\mathbb{E}\Big[ \int_{0}^{T}{| _{V^{*}}\langle \nabla  \cdot (\xi_2 (\nabla \Phi*\xi_1-\nabla \Phi*\xi_2)), \rho_1- \rho_2\rangle _{V}|dt }\Big]\\=:&I_1(t)+I_2(t).
\end{align*}
Next,  using (\ref{3.1}) and the Sobolev embedding $L^1(\mathcal{O})\subset{H_{\frac{m+1}{m}}^{-1}}(\mathcal{O}),~m+1>d$ to analyze  $I_1(t)$, we obtain that
\begin{align*}
I_1(t)&\leq\chi\mathbb{E}\Big[ \int_{0}^{T}|\int_\mathcal{O}(-\Delta)^{-1}\nabla  \cdot ((\xi_1-\xi_2) (\nabla \Phi*\xi_1))(\rho_1- \rho_2)dx|dt\Big]\\&\leq\chi\mathbb{E}\Big[ \int_{0}^{T}\|(-\Delta)^{-1}\nabla  \cdot ((\xi_1-\xi_2) (\nabla \Phi*\xi_1))\|_{L^{\frac{m+1}{m}}}\|\rho_1- \rho_2\|_{L^{m+1}}dt\Big]\\&\leq C \mathbb{E}\Big[ \int_{0}^{T}\|\nabla  \cdot ((\xi_1-\xi_2) (\nabla \Phi*\xi_1))\|_{H_{\frac{m+1}{m}}^{-2}}\|\rho_1- \rho_2\|_{L^{m+1}}dt\Big]\\&\leq C \mathbb{E}\Big[ \int_{0}^{T}\|(\xi_1-\xi_2) (\nabla \Phi*\xi_1)\|_{H_{\frac{m+1}{m}}^{-1}}\|\rho_1- \rho_2\|_{L^{m+1}}dt\Big]\\&\leq C \mathbb{E}\Big[ \int_{0}^{T}\|(\xi_1-\xi_2) (\nabla \Phi*\xi_1)\|_{L^1}\|\rho_1- \rho_2\|_{L^{m+1}}dt\Big]\\&\leq C \mathbb{E}\Big[ \int_{0}^{T}\|\xi_1-\xi_2\|_{L^2}\|\nabla \Phi*\xi_1\|_{L^2}\|\rho_1- \rho_2\|_{L^{m+1}}dt\Big]\\&\leq C \mathbb{E}\Big[ \int_{0}^{T}\|\xi_1-\xi_2\|_{L^2}\|\xi_1\|_{L^{m+1}}\|\rho_1- \rho_2\|_{L^{m+1}}dt\Big].
\end{align*}
Applying complex interpolation [\cite{2022Martingale}, p.21], we get for some  $\gamma \in\Big(0, \frac{1}{m}\Big)$, \begin{align*}
[H_{2}^{\gamma}(\mathcal{O}), H_{2}^{-1}(\mathcal{O})]_{\theta}=L^{2}(\mathcal{O})
\end{align*}
with  $\gamma(1-\theta)+\theta(-1)=0$. Thus, we know
\begin{align*}
\|\xi_{1}-\xi_{2}\|_{L^{2}} \leq\|\xi_{1}-\xi_{2}\|_{H_{2}^{-1}}^{\theta}\|\xi_{1}-\xi_{2}\|_{H_{2}^{\gamma}}^{1-\theta}. 
\end{align*}
This gives  $\theta=\frac{\gamma}{1+\gamma}<\frac{1}{1+m}$. Furthermore,  using the Young inequality  and the Jensen inequality,  we have that
\begin{align*}
	I_1(t)\leq&C(\varepsilon)\mathbb{E}\Big[ \int_{0}^{T}\|\xi_{1}-\xi_{2}\|_{H_{2}^{-1}}^{\frac{m+1}{m}\theta}\|\xi_{1}-\xi_{2}\|_{H_{2}^{\gamma}}^{\frac{m+1}{m}(1-\theta)} \|\xi_1\|_{L^{m+1}}^{\frac{m+1}{m}}dt\Big]+\varepsilon\mathbb{E}\Big[ \int_{0}^{T}\|\rho_1- \rho_2\|_{L^{m+1}}^{m+1}dt\Big]\\
    \leq&C(\varepsilon)\mathbb{E}\Big[ \Big(\sup _{0 \leq t \leq T}\|\xi_{1}-\xi_{2}\|_{H_{2}^{-1}}^{\frac{m+1}{m}\theta }\Big )\Big(\int_{0}^{T}\|\xi_{1}-\xi_{2}\|_{H_{2}^{\gamma}}^{\frac{m+1}{m}(1-\theta)}dt\Big)\Big(\sup _{0 \leq t \leq T}\|\xi_{1}\|_{L^{m+1}}^{\frac{m+1}{m}}\Big)\Big]\\&+\varepsilon\mathbb{E}\Big[ \int_{0}^{T}\|\rho_1- \rho_2\|_{L^{m+1}}^{m+1}dt\Big]\\
    \leq&C(\varepsilon,T)\mathbb{E}\Big[ \sup _{0 \leq t \leq T}\|\xi_{1}-\xi_{2}\|_{H_{2}^{-1}}^{\frac{m+1}{m}\theta p}\Big]^{\frac{1}{p}}\mathbb{E}\Big[\int_{0}^{T}\|\xi_{1}-\xi_{2}\|_{H_{2}^{\gamma}}^{\frac{m+1}{m}(1-\theta)q}dt\Big]^\frac{1}{q}\\
    &\cdot \mathbb{E}\Big[ \sup _{0 \leq t \leq T}\|\xi_{1}\|_{L^{m+1}}^{\frac{m+1}{m}r}\Big]^\frac{1}{r}+\varepsilon\mathbb{E}\Big[ \int_{0}^{T}\|\rho_1- \rho_2\|_{L^{m+1}}^{m+1}dt\Big]\\
    \leq&C(\varepsilon,T)\mathbb{E}\Big[ \sup _{0 \leq t \leq T}\|\xi_{1}-\xi_{2}\|_{H_{2}^{-1}}^{2}\Big]^{\frac{(m+1)\theta}{2m}}\mathbb{E}\Big[\int_{0}^{T}(\|\xi_{1}\|_{H_{2}^{\gamma}}+\|\xi_{2}\|_{H_{2}^{\gamma}})^{2m}dt\Big]^\frac{(m+1)(1-\theta)}{2m^2}\\
    &\cdot \mathbb{E}\Big[ \sup _{0 \leq t \leq T}\|\xi_{1}\|_{L^{m+1}}^{m+1}\Big]^\frac{1}{m}+\varepsilon\mathbb{E}\Big[ \int_{0}^{T}\|\rho_1- \rho_2\|_{L^{m+1}}^{m+1}dt\Big],
\end{align*}
where $p=2m$, $q=2$ and $r=\frac{2m}{m-1}$ satisfiy $\frac{1}{p}+\frac{1}{q}+\frac{1}{r}=1$. 
According to the proposition  of the reference  (\cite{2022Martingale}, p.10), we obtain 
\begin{align*}
\mathbb{E}\Big[\int_{0}^{T}&(\|\xi_{1}\|_{H_{2}^{\gamma}}+\|\xi_{2}\|_{H_{2}^{\gamma}})^{2m}\Big]dt\leq C\mathbb{E}\Big[\int_{0}^{T}(\|\xi_{1}\|_{H_{2m}^{\gamma}}^{2m}+\|\xi_{2}\|_{H_{2m}^{\gamma}}^{2m})\Big]dt\\&\leq C\mathbb{E}\Big[\int_{0}^{T}\||\xi_{1}|^{m-1}\nabla \xi_{1}\|_{L^{2}}^{2}\Big] d t+C\mathbb{E}\Big[\int_{0}^{T}\||\xi_{2}|^{m-1} \nabla \xi_{2}\|_{L^{2}}^{2} \Big]d t\leq CR_2.
\end{align*}
This gives
\begin{align*}
	I_1(t)&\leq C R_{2}^\frac{(m+1)(1-\theta)}{2m^2}R_{1}^{^\frac{1}{m}}\mathbb{E}\Big[ \sup _{0 \leq t \leq T}\|\xi_{1}-\xi_{2}\|_{H_{2}^{-1}}^{2}\Big]^{\theta}+\varepsilon\mathbb{E}\Big[ \int_{0}^{T}\|\rho_1- \rho_2\|_{L^{m+1}}^{m+1}dt\Big],
\end{align*}
where $\varepsilon=1$.
Next, we estimate $I_2(t)$ by Lemma \ref{lem2.4} and  the Sobolev embedding theorem
\begin{align*}
I_2(t)\leq &C\mathbb{E}\Big[ \int_{0}^{T}\|\nabla  \cdot( \xi_2 \nabla \Phi*(\xi_1-\xi_2))\| _{H_{2}^{-1}}\|\rho_1- \rho_2\| _{H_{2}^{-1}}dt\Big]\\
\leq& C \mathbb{E}\Big[ \int_{0}^{T}\|\xi_2\|_{L^{2}}\|\nabla \Phi(t)*(\xi_1-\xi_2)\|_{L^{\infty}}\|\rho_1- \rho_2\|_{H_{2}^{-1}}dt\Big]\\
\leq& C\mathbb{E}\Big[ \int_{0}^{T}\|\xi_2\|_{L^{m+1}}\|\xi_1-\xi_2\|_{L^{m+1}}\|\rho_1- \rho_2\|_{H_{2}^{-1}}dt\Big]\\
\leq& \varepsilon \mathbb{E}\Big[ \sup _{0 \leq t \leq T}\|\xi_2\|_{L^{m+1}}^{2}\int_{0}^{T}\|\xi_1-\xi_2\|_{L^{m+1}}^{2}dt\Big]+C(\varepsilon)\mathbb{E}\Big[ \int_{0}^{T}\|\rho_1- \rho_2\|_{H_{2}^{-1}}^{2}dt\Big]\\
\leq&\varepsilon \mathbb{E}\Big[\sup _{0 \leq t \leq T} \|\xi_2\|_{L^{m+1}}^{m+1}\Big]^\frac{2}{m-1}\mathbb{E}\Big[\int_{0}^{T}\|\xi_1-\xi_2\|_{L^{m+1}}^{m+1}dt\Big]^\frac{2}{m+1}\\
&+C(\varepsilon)\mathbb{E}\Big[ \int_{0}^{T}\|\rho_1- \rho_2\|_{H_{2}^{-1}}^{2}dt\Big]\\
\leq& \varepsilon R_{2}^\frac{2}{m-1}R_{1}^{^\frac{2}{m+1}}+C(\varepsilon)\mathbb{E}\Big[ \int_{0}^{T}\|\rho_1- \rho_2\|_{H_{2}^{-1}}^{2}dt\Big].
\end{align*}

We now estimate $J_2(t)$. By the Burkholder-Davis-Gundy inequality and the Young inequality, we get
\begin{align*}
J_2(t)&\leq C\mathbb{E}\Big[ \Big( \int_{0}^{T}\|\rho_1-\rho_2\|_{H_{2}^{-1}}^{2}
	\|\sigma(\rho_1- \rho_2)\|_{L_2(L^2(\mathcal{O}),H_{2}^{-1}(\mathcal{O}))}^{2}dt\Big) ^\frac{1}{2}\Big]\\
    \nonumber& \leq C\mathbb{E}\Big[\Big (\sup_{0 \leq t \leq T}\|\rho_1-\rho_2\|_{H_{2}^{-1}}^{2}\Big)^{\frac{1}{2}} \Big( \int_{0}^{T}\|\sigma(\rho_1-\rho_2)\|_{H_{2}^{-1}}^{2}dt\Big)^ \frac{1}{2}\Big]\\
    &\leq \frac{1}{4}\mathbb{E}\Big[\sup_{0 \leq t \leq T}\|\rho_1-\rho_2\|_{H_{2}^{-1}}^{2} \Big] + C\sigma^{2}\mathbb{E}\Big[  \int_{0}^{T}\|\rho_1-\rho_2\|_{H_{2}^{-1}}^{2}dt\Big].
\end{align*}

Combining the estimates of $J_1(t)$ and $J_2(t) $, we deduce that
\begin{align*}
&\frac{1}{4} \mathbb{E}\Big[\sup _{0 \leq t \leq T}\|(\rho_1- \rho_2)(t)\|_{H_{2}^{-1}}^{2}\Big]\\
\leq& R_{2}^\frac{(m+1)(1-\theta)}{2m^2}R_{1}^{^\frac{1}{m}}\mathbb{E}\Big[ \sup _{0 \leq t \leq T}\|\xi_{1}-\xi_{2}\|_{H_{2}^{-1}}^{2}\Big]^{\theta}+\varepsilon R_{2}^\frac{2}{m+1}R_{1}^{^\frac{2}{m+1}}+C(\sigma,\chi)\mathbb{E}\Big[ \int_{0}^{T}\|\rho_{1}- \rho_{2}\|_{H_{2}^{-1}}^{2}dt\Big].
\end{align*}
Futhermore, by Gronwall lemma, we get
\begin{align*}
\mathbb{E}\Big[\sup _{0 \leq t \leq T}\|\rho_1- \rho_2\|_{H_{2}^{-1}}^{2}\Big]\leq C(R_1,R_2,T)\mathbb{E}\Big[ \sup _{0 \leq t \leq T}\|\xi_{1}-\xi_{2}\|_{H_{2}^{-1}}^{2}\Big]^{\theta}.
\end{align*}
Taking $\delta=\theta$, we have (\ref{3.33}) holds. This completes the proof on the continuity of the solution operator $\mathcal{T}$.
\end{proof}
 Finally, we prove that $\mathcal{T}$ is a compact operator.  Using the Ascoli-Arzel\`a theorem, we give that $\mathcal{T}$ maps $X_T$ to a precompact set.
\begin{prop}\label{prop4}
For any $\rho_0\in L^{2}(\Omega ,\mathcal{F}_{0} ,H_{2}^{-1}(\mathcal{O} ))$ satisfying $\mathbb{E}\Big[\|\rho_{0}\|_{L^{m+1}}^{m+1}\Big]<\infty$, and all $R_{1} > 0$ and $R_{2} > 0$, it holds that
\item[(i)]  there exists $C = C(T,m)> 0$ such that for any $\xi \in X_{T}$, we have
\begin{align*}
\sup _{0 \leq t \leq T}\mathbb{E}\|\mathcal{T}(\xi)\|_{L^{m+1}}^{m+1}\leq C\mathbb{E}\Big[  \int_{0}^{T}\|\xi\|_{L^{m+1}}^{m+1}dt\Big]. 
\end{align*}
\item[(ii)] 
there exists a constant $C=C( T, m, R_{1}, R_{2})>0$  such that for any  $0\leq t_{1}<t_{2}\leq T$  and  $\xi\in X_{T}$
we have
$$
\mathbb{E}\|\mathcal{T} \xi(t_{1})-\mathcal{T} \xi(t_{2})\|_{H_{2}^{-1}}^{2} \leq C|t_{1}-t_{2}|.
$$
\end{prop}
\begin{proof}
Using It\^{o} formula to $\|\rho(t)- \rho_0 \|_{{H_{2}^{-1}}}^{2}$, we obtain
\begin{align*}
&\frac{1}{2}\| \rho(t)- \rho_0 \|_{{H_{2}^{-1}}}^{2}\\=&\int_{0}^{t}{_{V^{*}}\langle \Delta (|\rho(s)|^{m-1}\rho (s)),\rho(s)- \rho_0 \rangle_{V}}ds-\chi\int_{0}^{t}{ _{V^{*}}\langle \nabla  \cdot (\xi( \nabla \Phi*\xi)(s)), \rho(s)- \rho_0 \rangle_{V}}ds\\&+\frac{1}{2}\int_{0}^{t} \|\sigma\rho(s) \|_{L_2(L^2(\mathcal{O}),H_{2}^{-1}(\mathcal{O}))}^{2}+\int_{0}^{t}\langle\rho(s)- \rho_0 ,\sigma \rho(s) dW(s) \rangle_{H_{2}^{-1}}\\
\leq&\int_{0}^{t}{_{V^{*}}\langle \Delta( |\rho(s)|^{m-1}\rho (s)- \rho_0 ^{m}),\rho(s)- \rho_0 \rangle_{V}}ds+\int_{0}^{t}{_{V^{*}}\langle \Delta \rho_{0}^{m},\rho(s)- \rho_0 \rangle_{V}}ds\\&-\chi\int_{0}^{t}{ _{V^{*}}\langle \nabla  \cdot (\xi( \nabla \Phi*\xi)(s)), \rho(s)- \rho_0 \rangle_{V}}ds+\frac{1}{2}\int_{0}^{t} \|\sigma\rho(s) \|_{L_2(L^2(\mathcal{O}),H_{2}^{-1}(\mathcal{O}))}^{2}\\&+\int_{0}^{t}\langle\rho(s)- \rho_0 ,\sigma \rho(s) dW(s)\rangle_{H_{2}^{-1}}\\\leq&-\int_{0}^{t}\| \rho(s)- \rho_0\|_{L^{m+1}}^{m+1}ds+\int_{0}^{t}{_{V^{*}}\langle \Delta \rho_{0}^{m},\rho(s)- \rho_0 \rangle_{V}}ds\\&-\chi\int_{0}^{t}{ _{V^{*}}\langle \nabla  \cdot (\xi( \nabla \Phi*\xi)(s)), \rho(s)- \rho_0 \rangle_{V}}ds+\frac{1}{2}\sum_{k=1}^{\infty } \int_{0}^{t}\sigma ^{2}\mu_{k}^2\|\rho (s)e_{k}\|_{{H_{2}^{-1}}}^{2}ds\\&+\sum_{k=1}^{\infty }\int_{0}^{t}\langle \rho(s)- \rho_0 ,\sigma \rho(s)e_{k} \rangle_{H_{2}^{-1}}d\beta_{k}(s).
\end{align*}
This immediately gives
\begin{align}\label{3.118}
	\nonumber&\frac{1}{2}\| \rho(t)- \rho_0 \|_{{H_{2}^{-1}}}^{2}+\int_{0}^{t}\| \rho(s)- \rho_0\|_{L^{m+1}}^{m+1}ds\\
    \nonumber\leq&\int_{0}^{t}{_{V^{*}}\langle \Delta \rho_{0}^{m},\rho(s)- \rho_0 \rangle_{V}}ds-\chi\int_{0}^{t}{ _{V^{*}}\langle \nabla  \cdot (\xi( \nabla \Phi*\xi)(s)), \rho(s)- \rho_0 \rangle_{V}}ds\\
   \nonumber &+\sum_{k=1}^{\infty }\int_{0}^{t}\langle \rho(s)- \rho_0 ,\sigma \rho(s)e_{k} \rangle_{H_{2}^{-1}}d\beta_{k}(s)+C \int_{0}^{t}\|\rho (s)\|_{{H_{2}^{-1}}}^{2}ds\\=:&K_1(t)+K_2(t)+K_3(t)+C \int_{0}^{t}\|\rho (s)\|_{{H_{2}^{-1}}}^{2}ds.
\end{align}
For $K_1(s)$, applying the H\"{o}lder inequality in time gives that
\begin{align}\label{3.119}
	\nonumber\mathbb{E}\Big[\sup _{0 \leq s \leq t}K_1(s)\Big]&\leq\mathbb{E}\Big[ \int_{0}^{t}\|\rho(s)- \rho_0\|_{L^{m+1}}\|\rho_0\|_{L^{m+1}}^{m}ds\Big]\\
    \nonumber&\leq\mathbb{E}\Big[ \Big(\int_{0}^{t}\|\rho(s)- \rho_0\|_{L^{m+1}}^{m+1}ds\Big) ^{\frac{1}{m+1}}\Big(\int_{0}^{t}\|\rho_0\|_{L^{m+1}}^{m+1}ds\Big)^{\frac{m}{m+1}} \Big]  \\
    \nonumber&\leq \mathbb{E}\Big[ \Big( \int_{0}^{t}\|\rho(s)- \rho_0\|_{L^{m+1}}^{m+1}ds\Big) ^{\frac{1}{m+1}}t^{\frac{m}{m+1}}\|\rho_0\|_{L^{m+1}}^{m} \Big]\\
    &\leq\varepsilon_1 \mathbb{E}\Big[ \int_{0}^{t}\|\rho(s)- \rho_0\|_{L^{m+1}}^{m+1}ds\Big]+C(\varepsilon_1,m)t\mathbb{E}\Big[ \|\rho_0\|_{L^{m+1}}^{m+1}\Big].
\end{align}

In order to calculate $K_2(s)$,  by (\ref{3.3}) and  the Young inequality, we get
\begin{align}\label{3.20}
\nonumber\mathbb{E}\Big[\sup _{0 \leq s \leq t}K_2(s)\Big]&\leq C(\chi)\mathbb{E}\Big[ \int_{0}^{t}\|\rho(s)- \rho_0\|_{{H_{2}^{-1}}}\|\xi\|_{L^{m+1}}^{2}ds\Big]\\&\leq\varepsilon_2\mathbb{E}\Big[ \int_{0}^{t}\|\rho(s)- \rho_0\|_{{H_{2}^{-1}}}^{2}ds\Big]+C(\varepsilon_2)t\mathbb{E}\Big[ \sup _{0 \leq s \leq t}\|\xi\|_{{L^{m+1}}}^{m+1}\Big]^\frac{4}{m+1}.
\end{align}

We now estimate $K_3(s)$, by the Burkholder-Davis-Gundy inequality 
\begin{align}\label{3.21}
\nonumber\mathbb{E}\Big[\sup _{0 \leq s \leq t}K_3(s)\Big]&\leq C\mathbb{E}\Big[ \Big( \int_{0}^{t}\|\rho(s)-\rho_0\|_{H_{2}^{-1}}^{2}\|\sigma\rho(s)\|_{L_2(L^2(\mathcal{O}),H_{2}^{-1}(\mathcal{O}))}^{2}ds\Big) ^\frac{1}{2}\Big]\\\nonumber& \leq C\mathbb{E}\Big[\Big (\sup_{0 \leq s \leq t}\|\rho(s)-\rho_0\|_{H_{2}^{-1}}^2\Big)^\frac{1}{2}\Big( \int_{0}^{t}\|\sigma\mu_k\rho(s)\|_{H_{2}^{-1}}^{2}ds\Big)^ \frac{1}{2}\Big]\\&\leq \frac{1}{4}\mathbb{E}\Big[\sup_{0 \leq s \leq t}\|\rho(s)-\rho_0\|_{H_{2}^{-1}}^{2} \Big] + C\sigma^{2}\mu_k^2t\mathbb{E}\Big[  \sup _{0 \leq s \leq t}\|\rho(s)\|_{H_{2}^{-1}}^{2}\Big].
\end{align}
By (\ref{3.119})-(\ref{3.21}), taking $\varepsilon_1,\varepsilon_2$ small enough and for any  $t\in[0,T]$, there exist a constant $C$ such that
\begin{align*}
&\frac{1}{4}\mathbb{E}\Big[\sup _{0 \leq s \leq t}\| \rho(s)- \rho_0 \|_{{H_{2}^{-1}}}^{2}\Big]+\mathbb{E}\Big[ \int_{0}^{t}\|\rho(s)- \rho_0\|_{L^{m+1}}^{m+1}ds\Big]\\\leq& C(\varepsilon_1,m)t\mathbb{E}\Big[ \|\rho_0\|_{L^{m+1}}^{m+1}\Big]+C(\varepsilon_2)t\mathbb{E}\Big[ \sup _{0 \leq s \leq t}\|\xi\|_{{L^{m+1}}}^{m+1}\Big]^\frac{4}{m+1}+C\sigma^{2}t\mathbb{E}\Big[  \sup _{0 \leq s \leq t}\|\rho(s)\|_{H_{2}^{-1}}^{2}\Big]\\\leq& C_1t\mathbb{E}\Big[ \|\rho_0\|_{L^{m+1}}^{m+1}\Big]+C_2tR_{2}^\frac{4}{m+1}+C_3tR_1\leq C( T, m, R_{1}, R_{2})t.
\end{align*} 
 This ﬁnishes the proof of the proposition.
\end{proof}
Basing on Proposition 3.1-3.4,  we apply the Schaduder fixed point theorem to get the existence of marginal solutions.

	\nocite{*}
	\bibliography{refer} %你存放bibtex格式参考文献的文件名
	\bibliographystyle{plain}  %你的参考文献模板

\end{document}